\documentclass[14pt, dvipsnames]{extarticle}
\usepackage[left=2cm, top=2cm, right=2cm, bottom=20mm]{geometry} 

\usepackage[T1]{fontenc}
\usepackage[cmintegrals,cmbraces]{newtxmath}
\usepackage[lining]{ebgaramond}
\usepackage{ebgaramond-maths}
\usepackage{euscript}
\usepackage{xcolor}
\usepackage{graphicx}
\usepackage[figurename=Fig.]{caption}

\usepackage[tracking = true, letterspace = 50]{microtype}
\SetTracking{encoding=T1, shape=sc}{50}
\SetTracking{encoding=T1}{50}

\usepackage{amsthm} 
\usepackage{hyperref} 
\usepackage[hyphenbreaks]{breakurl}

\theoremstyle{plain}
\newtheorem{theorem}{Theorem}

\newtheorem{lemma}{Lemma}

\newtheorem{prop}{Proposition}

\newtheorem*{remark*}{Remark}

\theoremstyle{definition}
\newtheorem{defi}{Definition}

\newtheorem{example}{\bf Example}

\theoremstyle{remark}
\newtheorem{remark}{\bf Remark}

\usepackage{enumerate}

\usepackage{etoolbox}
\patchcmd{\thmhead}{(#3)}{#3}{}{}
\makeatletter
\renewenvironment{proof}[1][\proofname]{
  \par\pushQED{\qed}\normalfont
  \topsep6\p@\@plus6\p@\relax
  \trivlist\item[\hskip\labelsep\bfseries\itshape #1\@addpunct{.}]
}{\popQED\endtrivlist\@endpefalse}
\makeatother

\let\phi\varphi

\DeclareMathOperator{\disc}{\mathrm{disc}}

\DeclareMathOperator{\Q}{\mathbb{Q}}

\DeclareMathOperator{\E}{\EuScript{E}}

\DeclareMathOperator{\Tt}{\mathbb{T}}

\DeclareMathOperator{\G}{\mathbb{G}}

\DeclareMathOperator{\Sym}{\mathrm{Sym}}
\DeclareMathOperator{\nGrp}{\mathit{n}-\mathcal{G}\mathrm{rp}}
\DeclareMathOperator{\inv}{\mathrm{inv}}
\DeclareMathOperator{\Aut}{\mathrm{Aut}}

\DeclareMathOperator{\Z}{\mathbb{Z}}
\DeclareMathOperator{\Cc}{\mathbb{C}}

\DeclareMathOperator{\R}{\mathbb{R}}

\DeclareMathOperator{\Oo}{\EuScript{O}}

\DeclareMathOperator{\diag}{\mathrm{diag}}

\DeclareMathOperator{\Mat}{\mathrm{Mat}}
\DeclareMathOperator{\Tr}{\mathrm{Tr}}
\DeclareMathOperator{\eps}{\varepsilon}

\DeclareMathOperator{\M}{\mathbb{M}}
\DeclareMathOperator{\mult}{\mathrm{mult}}

\DeclareMathOperator{\node}{\mathrm{node}}

\DeclareMathOperator{\cusp}{\mathrm{cusp}}

\DeclareMathOperator{\W}{\mathbb{W}}
\DeclareMathOperator{\CP}{\mathbb{C}P}

\makeatletter
  \DeclareSymbolFont{ntxletters}{OML}{ntxmi}{m}{it}
  \SetSymbolFont{ntxletters}{bold}{OML}{ntxmi}{b}{it}
  \re@DeclareMathSymbol{\leftharpoonup}{\mathrel}{ntxletters}{"28}
  \re@DeclareMathSymbol{\leftharpoondown}{\mathrel}{ntxletters}{"29}
  \re@DeclareMathSymbol{\rightharpoonup}{\mathrel}{ntxletters}{"2A}
  \re@DeclareMathSymbol{\rightharpoondown}{\mathrel}{ntxletters}{"2B}
  \re@DeclareMathSymbol{\triangleleft}{\mathbin}{ntxletters}{"2F}
  \re@DeclareMathSymbol{\triangleright}{\mathbin}{ntxletters}{"2E}
  \re@DeclareMathSymbol{\partial}{\mathord}{ntxletters}{"40}
  \re@DeclareMathSymbol{\flat}{\mathord}{ntxletters}{"5B}
  \re@DeclareMathSymbol{\natural}{\mathord}{ntxletters}{"5C}
  \re@DeclareMathSymbol{\star}{\mathbin}{ntxletters}{"3F}
  \re@DeclareMathSymbol{\smile}{\mathrel}{ntxletters}{"5E}
  \re@DeclareMathSymbol{\frown}{\mathrel}{ntxletters}{"5F}
  \re@DeclareMathSymbol{\sharp}{\mathord}{ntxletters}{"5D}
  \re@DeclareMathAccent{\vec}{\mathord}{ntxletters}{"7E}
\makeatother

\renewcommand{\epsilon}{\varepsilon}

\DeclareFontFamily{U}{EBSUB}{}
\DeclareFontShape{U}{EBSUB}{m}{it}{<-> EBGaramond-Italic-tlf-lgr}{}
\DeclareFontFamily{U}{EBSUB}{}
\DeclareSymbolFont{EBSUB}{U}{EBSUB}{m}{it}

\DeclareMathSymbol{\mu}{\mathord}{EBSUB}{`m}
\DeclareMathSymbol{\varDelta}{\mathord}{EBSUB}{`D}
\DeclareMathSymbol{\varOmega}{\mathord}{EBSUB}{`W}

 \hypersetup{
    colorlinks,
    linkcolor={red!50!black},
    citecolor={blue!50!black},
    urlcolor={blue!80!black}
}

\usepackage{stackrel}
\usepackage[all]{xy}

\usepackage{bm}

\usepackage{indentfirst}

\usepackage{tocloft}

\title{\bf Algebraic $\bm{n}$-Valued Monoids on $\bm{\mathbb{C}P^1}$,\\ 
Discriminants and Projective Duality}

\date{}

\author{Victor Buchstaber and Mikhail Kornev}

\begin{document}

\maketitle

\begin{abstract}
In this work, we establish connections between the theory of algebraic $n$-valued monoids and groups and the theories of discriminants and projective duality. We show that the composition of projective duality followed by the Möbius transformation $z\mapsto 1/z$ defines a shift operation $\mathbb{M}_n(\mathbb{C}P^1)\mapsto \mathbb{M}_{n-1}(\mathbb{C}P^1)$ in the family of algebraic $n$-valued coset monoids $\{\mathbb{M}_{n}(\mathbb{C}P^1)\}_{n\in\mathbb{N}}$. We also show that projective duality sends each Fermat curve $x^n+y^n=z^n$ $(n\ge 2)$ to the curve $p_{n-1}(z^n; x^n, y^n)=0$, where the polynomial $p_n(z;x,y)$ defines the addition law in the monoid $\M_n(\Cc\!P^1)$. We solve the problem of describing coset $n$-valued addition laws constructed from cubic curves. As a corollary, we obtain that all such addition laws are given by polynomials, whereas the addition laws of formal groups on general cubic curves are given by series.
\end{abstract}

\tableofcontents

\thispagestyle{empty}

\section{Introduction}\label{sec:intro}\label{intro}

In \cite{Kontsevich_type_polynomials, Gaiur}, the addition laws of algebraic $n$-valued groups on $\Cc\!P^1$ \cite{Buchstaber} were expressed in terms of discriminants. In the present work, we develop the connection between the theory of algebraic $n$-valued monoids and $n$-valued groups on $\Cc\!P^1$ and the theories of discriminants and projective duality \cite{Gelfand}. We use algebro-geometric methods without invoking the theory of elliptic functions.

An {\it algebraic $n$-valued monoid} is an algebraic variety $X$ with an associative $n$-valued multiplication defined by a rational morphism $X\times X\to \Sym^n(X)$ and with a neutral element (identity) $e\in X$, i.e. $x\ast e$ and $e\ast x$ exist and the equalities $x\ast e = e\ast x = [x, x, ..., x]$ hold for every $x\in X$. An {\it algebraic $n$-valued group} is an algebraic $n$-valued monoid on $X$ with a rational inversion morphism $\inv: X\to X$ such that, whenever $\inv(x)$ exists, the following conditions hold: $x\ast\inv(x)$ and $\inv(x)\ast x$ exist, and $e\in x\ast\inv(x)$, $e\in \inv(x)\ast x$.

An algebraic $n$-valued monoid (or group) on $X$ is called {\it regular} if the $n$-valued multiplication $X\times X\to \Sym^n(X)$ is defined on the whole variety $X\times X$ (and the map $\inv$ is defined everywhere on $X$ in the case of a group). An $n$-valued group $X$ is called {\it involutive} if the map $\inv$ is uniquely defined and $\inv(x) = x$ for every $x\in X$. 

Let
$$\delta_{\boldsymbol{a}} = \disc_t(t^3 + a_1t^2 + a_2t + a_3)$$
be the discriminant. It was shown in \cite[Theorem 6.3]{Buchstaber_Veselov19} that, for any choice of complex parameters $\boldsymbol{a} = (a_1, a_2, a_3)$, the Buchstaber polynomials
\begin{equation}\label{B}
B_{\boldsymbol{a}}(z; x, y) = (x + y + z - a_2xyz)^2 - 4(1 + a_3xyz)(xy + yz + xz + a_1xyz)
\end{equation}
define the structure of the universal symmetric $2$-algebraic $2$-valued group $\G_{\Cc}(B_{\boldsymbol{a}})$ on $\Cc$, with multiplication
$$x\ast y = \{\,z\mid B_{\boldsymbol{a}}(z; x, y)= 0\,\},$$
zero element $0$, and inverse $\inv(x) = x$. According to the recent paper \cite[Theorem 4.7]{Kontsevich_type_polynomials}, for $\delta_{\boldsymbol{a}}\neq 0$, the law \eqref{B} defines on $\Cc\!P^1$ the structure of a regular $2$-valued algebraic group.   

As was observed in \cite[Theorem 4.7]{Kontsevich_type_polynomials}, the following equality holds:
$$D_{\boldsymbol{a}}(z; x, y) = (xyz)^2B_{\boldsymbol{a}}(-1/z; -1/x, -1/y),$$
where
$$D_{\boldsymbol{a}}(z; x, y) = \disc_t\bigl(t^3 + a_1 t^2 + a_2 t + a_3 - (t - x)(t - y)(t - z)\bigr)$$
is the generalized Kontsevich polynomial. 

Denote by $\G_n(\Cc)$ the coset \cite[Theorem 1]{Buchstaber} $n$-valued algebraic group on $\Cc$ with zero element $0$, inverse $\inv(x) = (-1)^nx$, and multiplication
$$x\ast y = [z\mid p_n(z; (-1)^nx, (-1)^ny) = 0],$$
where
\begin{equation}\label{polynomial_p_n_intro}
p_n(z; x, y) = \prod\limits_{r, s = 1}^{n}(\sqrt[n]{z} + \varepsilon^r\sqrt[n]{x} + \varepsilon^s\sqrt[n]{y})
\end{equation}
is a symmetric polynomial with integer coefficients, $\varepsilon = e^{2\pi i/n}$, and a fixed branch of $\sqrt[n]{-}$ is chosen. The group $\G_n(\Cc)$ was first considered in \cite{Buchstaber_Novikov}. 

In \S\ref{algebraic_nval_monoids_and_groups} of the present paper, we give the basic constructions, examples, and properties of algebraic $n$-valued monoids and groups.

In Theorems \ref{2-valued_group}, \ref{3-valued_group}, \ref{4-valued_group}, and \ref{6-valued_group}, we describe $n$-valued addition laws for $n=2$, $3$, $4$, and $6$ which are explicitly given by polynomials and are obtained by coset constructions from geometric addition laws on elliptic curves. In this connection, we note that in \cite{Bunkova_Buchstaber11} a single-valued addition law on elliptic curves in the general Weierstrass model was presented using Tate uniformization. This law is given by formal series. One consequence of our work is the fact that the $n$-valued laws corresponding to elliptic curves are given by polynomials.

The coset construction for groups {\normalfont \cite[Theorem 1]{Buchstaber}} carries over without difficulty to the case of monoids. We shall call an $n$-valued monoid $M_{H}$ constructed from a $1$-valued monoid $M$ and a subgroup $H$ of order $n$ of the group $\Aut(M)$ a {\it coset $n$-valued monoid}.  

Theorem \ref{2-valued_group} describes the classification of all $2$-valued coset groups and monoids on $\CP^1$ obtained by the coset construction from the group of points of an elliptic curve with respect to an involution. Although the content of part {\bf{(i)}} is not new, its method of proof is. This result was first obtained in \cite[Theorem 4.7]{Kontsevich_type_polynomials} using the theory of elliptic functions. According to this result, for $\delta_{\boldsymbol{a}}\neq 0$ (and only in this case), the universal $2$-valued group law $\G_{\Cc}(B_{\boldsymbol{a}})$ obtained in \cite[Theorem 6.3]{Buchstaber_Veselov19} extends to a $2$-valued coset algebraic regular involutive group $\G_{\Cc\!P^1}(B_{\boldsymbol{a}})$ on $\Cc\!P^1$ with zero element $0$ and multiplication $\mu_{\boldsymbol{a}}$. In \cite[Theorem 5.1]{Kontsevich_type_polynomials}, it was shown that the Möbius transformation $x\mapsto -1/x$, $y\mapsto -1/y$, $z\mapsto -1/z$ takes the group $\G_{\Cc\!P^1}(B_{\boldsymbol{a}})$ to an isomorphic group $\G_{\Cc\!P^1}(D_{\boldsymbol{a}})$ with zero element $\infty$ and multiplication defined by the Kontsevich polynomial $D_{\boldsymbol{a}}(-z; -x, -y)$. The group $\G_{\Cc\!P^1}(D_{\boldsymbol{a}})$ coincides with the coset group $\E_{\langle\sigma\rangle}$, where
\begin{equation}\label{E_cubic}
\E = \{\,\eta^2 = \zeta^3 + a_1\zeta^2 + a_2\zeta + a_3\,\}
\end{equation}
is an elliptic curve and
$$\begin{aligned}
\sigma: (\zeta, \eta)&\mapsto (\zeta, -\eta)\\
\infty&\mapsto\infty
\end{aligned}$$
is an involution. Part {\bf (ii)} of Theorem \ref{2-valued_group} is new and states that the groups $\G_{\Cc\!P^1}(B_{\boldsymbol{a}})$ are classified by the $j$-invariant of the elliptic curve $\E$.

We shall call an element $w$ of an $n$-valued monoid (group) $\M$ {\it iterating} (for $n = 2$, {\it doubling}) if, for every $m\in\M$, the multisets $m\ast w$ and $w\ast m$ contain points of multiplicity at least $2$. For $n = 2$, the set of doubling points forms a $2$-valued diagonal submonoid (subgroup). If an $n$-valued algebraic group is given in some chart $U$ by the roots in $z$ of a certain polynomial $P(z; x, y)$, then an element $y$ is iterating if and only if $y$ is a root of the discriminant $\disc_z(P(z; x, y))$ for every $x\in U$. Thus, the discriminant $\disc_z(P(z; x, y))$ of the law $P(z; x, y)$ carries important information about the structure of the $n$-valued algebraic monoid. In Theorem \ref{doubling_group_for_elliptic_curves}, it is shown that for $\delta_{\boldsymbol{a}}\neq 0$, the group of doubling points of the $2$-valued group $\G(B_{\boldsymbol{a}})$ is isomorphic to the Klein four-group $\Z/2\times\Z/2$. {By Proposition \ref{iterating_3-val}, the set of iterating elements of the $3$-valued group $\G_{3, c}(\Cc\!P^1)$ is, as a $3$-valued subgroup, the diagonal construction of a $1$-valued group isomorphic to $\Z/3$. By Proposition \ref{iterating_4-val}, the iterating elements of the $4$-valued group $\G_{4, b}(\Cc\!P^1)$ form a $4$-valued subgroup isomorphic to the group $\diag_2((\Z/4)/{\sigma})$, where $\diag_2$ denotes the double diagonal and $(\Z/4)/{\sigma}$ is the two-valued coset group constructed from the additive group $\Z/4$ and the involution $\sigma:g\mapsto -g$ for all $g\in\Z/4$.}

As already noted, Theorems \ref{3-valued_group}, \ref{4-valued_group}, and \ref{6-valued_group} explicitly describe the polynomials defining multiplication in all possible, up to isomorphism, coset $3$-, $4$-, and $6$-valued groups $\G_{3, c}(\Cc\!P^1)$, $\G_{4, b}(\Cc\!P^1)$, and $\G_{6, c}(\Cc\!P^1)$ on $\Cc\!P^1$ modeled by elliptic curves and automorphisms. They correspond to equianharmonic ($j = 0$, $\Aut \E(\Cc) \cong \Z/6$) and harmonic ($j = 1728$, $\Aut \E(\Cc)\cong \Z/4$) elliptic curves $\E$. 

In the nodal case of the cubic $\E$, according to Theorem \ref{nodal_monoids}, the coset group $\G_{\Cc\!P^1}(D_{\boldsymbol{a}})$ becomes, up to isomorphism, the coset monoid $\M_{\node}(\Cc\!P^1)$.

In the cuspidal case of the cubic $\E$, according to Theorem \ref{cusp_monoids}, the coset groups
$$\G_{\Cc\!P^1}(D_{\boldsymbol{a}}),\ \G_{3,c}(\Cc\!P^1),\ \G_{4, b}(\Cc\!P^1),\ \G_{6, c}(\Cc\!P^1)$$
become, up to isomorphism, the coset monoids
$$\M_2(\Cc\!P^1),\ \M_3(\Cc\!P^1),\ \M_4(\Cc\!P^1),\ \M_6(\Cc\!P^1),$$
respectively.

{In \cite{BK}, the notion of \textit{symmetric $n$-algebraic $n$-valued group} was introduced. For $n = 3$, the full classification was obtained. We note that, to date, no complete classification of $n$-algebraic $n$-valued groups is known for $n>3$.} 

Theorem \ref{Fermat_curve_theorem} states that, under projective duality, the Fermat curve $\{\,x^n+y^n=z^n\,\}$ is sent to the curve $\{\,p_{n-1}(z^n;x^n,y^n)=0\,\}$.

Theorem \ref{M_n_monoid_structure} shows that the structure of the algebraic $n$-valued coset group $\G_n(\Cc)$ extends only to the structure of an algebraic $n$-valued monoid $\M_n(\Cc\!P^1)$ on $\Cc\!P^1$. Here the point $\infty$ is absorbing, that is,
\[
\infty\ast x=x\ast\infty=[\infty,\infty,\dots,\infty]
\quad\text{for every }x\in \Cc\!P^1\setminus\{\,\infty\,\}.
\]

For every positive integer $n$, the polynomial $p_n$ defines a curve
$$X_n=\{\,p_n(z;x,y)=0\,\}$$
in $\Cc\!P^2$. By Theorem \ref{shift_duality_monoids}, under projective duality the curve $X_n$ $(n\ge 2)$ is sent to
\[
X_n^{\vee}=\{\,(uvw)^{n-1}p_{n-1}(1/w;1/u,1/v)=0\,\}\subset (\Cc\!P^2)^{\ast},
\]
and the composition of the duality $X_n\mapsto X_n^{\vee}$ with the subsequent Möbius transformation $(u,v,w)\mapsto (1/u,1/v,1/w)$ defines the shift operation $\M_n(\Cc\!P^1)\mapsto \M_{n-1}(\Cc\!P^1)$ in the family of algebraic $n$-valued monoids. It follows from the Plücker formulas \cite[Proposition 2.4]{Gelfand} that if $X$ is a smooth curve of degree $n$, then the curve $X^{\vee}$ has degree $n(n-1)$. In our case, $X_n$ and $X_n^{\vee}$ are singular for $n \ge 3$, with $\deg X_n = n$ and $\deg X_n^{\vee} = (n -1)^2$. The curves $X_2$ and $X_2^{\vee}$ are nonsingular; see Example \ref{X_2}. 

Recall that, according to \cite[Theorem 2.3]{Gaiur}, the polynomial $p_n(z;x,y)$ and the discriminant $\disc_t(P_{x,y,z}(t))$ of the polynomial
$$P_{x,y,z}(t)=(-1)^n x\,t^{\,n-1}(1+t)^{n-1}+(-1)^n y(1+t)^{n-1}-t^{\,n-1}z$$
with respect to the variable $t$ satisfy the equality
\begin{equation}\label{disc_equality}
(-1)^n (n-1)^{2(n-1)} (xyz)^{n-2} p_n(z;x,y)=\disc_t\bigl(P_{x,y,z}(t)\bigr).
\end{equation}
For an explicit proof, see \cite[Theorem 8]{BK}. Theorems \ref{Fermat_curve_theorem}, \ref{M_n_monoid_structure}, and \ref{shift_duality_monoids} explain \eqref{disc_equality} using the theory of \cite{Gelfand}, which relates discriminants and projective duality.

The iterations of the $n$-valued addition in $\G_n(\Cc)$ are defined by the symmetric polynomials
\[
p_{n,m}(z;\bm{x})=\prod_{k_1,\dots,k_m=1}^{n}\bigl(\sqrt[n]{z}+\varepsilon^{k_1}\sqrt[n]{x_1}+\dots+\eps^{k_m}\sqrt[n]{x_m}\bigr),
\]
which arise, for example, in connection with Picard--Fuchs differential equations \cite[Section 3]{Gaiur}. Denote by $\Oo_{n,m}(\Cc\!P^1)$ the variety $\Cc\!P^1$ equipped with this operation. Let $X_{n,m}=\{\,p_{n,m}=0\,\}$ be a hypersurface in $\Cc\!P^m$, and put
$$P_{n,m}(w; \bm{u})=(u_1\cdots u_m w)^{n-1}\, p_{n-1}(w^{-1};u_1^{-1},\dots,u_m^{-1}).$$
Then Theorem \ref{theorem_3} states that the composition of duality, for $m\ge 2$ and $n\ge 2$,
$$X_{n,m}\mapsto X_{n,m}^{\vee}=\{\,P_{n,m}=0\,\}\subset (\Cc\!P^m)^{\ast}$$
with the subsequent Möbius transformation
$$(u_1,\dots,u_m,w)\mapsto (1/u_1,\dots,1/u_m,1/w)$$
defines the shift operation
$$\Oo_{n,m}(\Cc\!P^1)\mapsto \Oo_{n-1,m}(\Cc\!P^1)$$
in the family of $m$-ary $n^{m-1}$-valued algebraic structures $\Oo_{n,m}(\Cc\!P^1)$.

Theorem \ref{F_nm_hypersurfaces} is an iterated analogue of Theorem \ref{Fermat_curve_theorem}. It gives a concrete realization
$$F_{n,m}^{\vee} = \{\, p_{n-1, m}(w^n; u_1^n, ..., u_m^n) = 0\,\}$$
as an explicit equation for the hypersurface dual to the Fermat hypersurface
$$F_{n, m} = \{\, x_1^n + ... + x_m^n = z^n\, \}.$$

Section \ref{projective_duality_and_the_family_of_monoids} is devoted to these results.

The main result of Section \ref{p_n_laws_and_field_extension_discriminants}, Proposition \ref{p_n_as_field_extension_discriminant}, describes the polynomial $p_n(z; x, y)$ as the discriminant of a certain field extension.

The authors are grateful to Vladimir Rubtsov for helpful discussions during the preparation of this work.

\section
[Basic Concepts of the Theory of $\texorpdfstring{\boldsymbol{n}}{n}$-Valued Groups]
{Basic Concepts of the Theory of $\boldsymbol{n}$-Valued Groups}
\label{Basic_concepts_of_n_val}

In this section we give several definitions and examples, following \cite{Buchstaber}.

\begin{defi}\label{nval_group_def}
An $n$-valued multiplication on a set $X$ is a map

$$\mu: X\times X \to \Sym^n(X)\,:\,\mu(x,y)=x*y,$$ where $\Sym^n(X):=X^n/\Sigma_n$ is the $n$-th symmetric power of the set $X$ (i.e. the set consisting of multisets $[z_1, ..., z_n]$), such that the following conditions hold:

\begin{itemize}

\item {\it Associativity.} The following $n^2$-multisets
\begin{gather*}
[x*w\mid w\in y \ast z],\\
[w\ast z\mid w\in x*y]
\end{gather*} 
coincide.

\item {\it Identity.} There exists an element $e\in X$ such that $e*x=x*e=[x,x,\dots,x]$\; for every\; $x\in X$.

\item {\it Inverse element.} There is a map $\mathrm{inv}\colon X\to X$ with the properties
$$e\in \mathrm{inv}(x)* x,\ e \in x*\mathrm{inv}(x)$$ for every $x\in X$.
\end{itemize}
\end{defi}

\begin{defi}
If a set $X$ is equipped with an $n$-valued multiplication, then $X$ is called an {\it $n$-valued group}.
\end{defi}

\begin{example}
Every $1$-valued group $G$ is an ordinary group.   
\end{example}

\begin{example}\label{G_n_defi}
On the set $\Cc$ of complex numbers there is a structure of an $n$-valued group $\G_n(\Cc)$. Consider the operation $\ast: \Cc\times\Cc\to \Sym^n(\Cc)$ which assigns to every pair of complex numbers $x$ and $y$ the multiset
\begin{equation}\label{operation2}
x\ast y = [(\sqrt[n]{x} + \epsilon^r\sqrt[n]{y})^n\mid r = 1, ..., n]
\end{equation}
of $n$ complex numbers, where $\epsilon$ is some primitive $n$-th root of unity and $\sqrt[n]{\cdot}$ denotes some fixed branch of the root. The operation $\ast$, the identity $0$, and the inverse $\inv(x) = (-1)^nx$ endow the set of complex numbers $\Cc$ with the structure of an $n$-valued group.
\end{example}

\begin{defi}
An $n$-valued group on a set $X$ is called {\it commutative} if $x_1\ast x_2 = x_2\ast x_1$ for all $x_1, x_2\in X$.
\end{defi}

\begin{defi}
A map $f: X\to Y$ between $n$-valued groups is called a {\it homomorphism} if 

\begin{itemize}

\item $f(e_X)=e_Y$.

\item $f(\mathrm{inv}_X(x))=\mathrm{inv}_Y(f(x))$ for every $x\in X$.

\item $\mu_{Y}(f(x),f(y))=(f)^n\mu_X(x,y)$ for all $x, y\in X$.

\end{itemize}

That is, the following diagram is commutative:

\[
\xymatrix{
    X\times X \ar[r]^{\mu_X} \ar[d]_{f\times f} & \Sym^n(X) \ar[d]^{\Sym^n(f)} \\
    Y\times Y \ar[r]       & \Sym^n(Y) }
\]

A bijective homomorphism of $n$-valued groups is called an {\it $n$-isomorphism}.

\end{defi}

Thus, the class of $n$-valued groups forms a category $\nGrp$.

There is a construction of $n$-valued groups \cite{Buchstaber} in which, for a given ordinary ($1$-valued) group $G$ with multiplication $\mu_0$, identity $e_G$, and inverse $\inv_G(u) = u^{-1}$, and for a given finite subgroup $H$ of the automorphism group $\Aut(G)$, one considers the set of orbits $X := G/H$ of the group $G$ under the action of $H$, with projection $\pi: G\to X$, and defines an $n$-valued multiplication
$$\mu: X\times X\to \Sym^n(X),$$
by the formula
\begin{equation}\label{coset}
\mu(x, y) = [ \pi(\mu(u, h(v))) \ | \ h\in H ],
\end{equation}
where $u\in\pi^{-1}(x)$ and $v\in\pi^{-1}(y)$.

Let us note that non-isomorphic pairs $(G_1, H_1)$ and $(G_2, H_2)$ may give rise to isomorphic $n$-valued groups \cite[Proposition 3.2]{BuchRees}.

\begin{prop}[(\textnormal{Coset construction}, \cite{Buchstaber})]\label{coset_theorem}
The product $\mu$ from {\normalfont (\ref{coset})} defines an $n$-valued multiplication on the orbit space $X = G/H$ with identity $e_X = \pi(e_G)$ and inverse $\inv_X(x) = \pi(\inv_G(u))$, where $u\in\pi^{-1}(x)$.
\end{prop}

\begin{defi}\label{coset_defi}
For each group $G$ and each subgroup $H$ of order $n$ of the group $\Aut(G)$, we shall call the $n$-valued group on the orbit set $G/H$ from Proposition \ref{coset_theorem} a {\it coset group} or a {\it coset construction}.
\end{defi}

Let us note that the uniqueness of the inverse-element map does not follow from the axioms of an $n$-valued multiplication. However, in the case of the coset construction, the inverse is unique.

\begin{defi}
An {\it $n$-valued dynamics} $T$ on a space $X$ is a map $T: X\to \Sym^n(X)$.
\end{defi} 

One may think of $X$ as a certain state space. Then an $n$-valued dynamics $T$ determines the possible states $T(x) = [x_1, ..., x_n]$ at time $t+1$ as a function of the state of the element $x$ at time $t$.

\begin{defi}\label{group_action}
An {\it action of an $n$-valued group $G$ on a space} $X$ is a map
$$\circ: G\times X\to \Sym^n(X),\ g\circ x = [x_1, ..., x_n],$$
such that

\begin{itemize}

 \item for any elements $g_1$, $g_2\in G$ and any $x\in X$, the following $n^2$-multisets coincide:
 $$g_1\ast(g_2\circ x) = [g_1\circ x_1, ..., g_1\circ x_n]$$
 and
 $$(g_1\ast g_2)\circ x = [h_1\circ x, ..., h_n\circ x],$$
 where $g_2\circ x = [x_1, ..., x_n]$ and $g_1\ast g_2 = [h_1, ..., h_n]$,
 
 \item $e\circ x = [x, ..., x]$ for every $x\in X$. 
 
 \end{itemize}
 
\end{defi}

\begin{example}
Every $n$-valued group $G$ acts on itself by left shifts, just as in the case of single-valued groups. As was shown in the recent paper \cite{Posadskiy}, there exists a strongly reversible continuous dynamics on $\Cc$ which is not generated by a continuous action of any two-valued group.
\end{example}

In \cite{Buchstaber_Veselov}, an integrability condition for an $n$-valued dynamics on a space $X$ was formulated in terms of a representation of a certain $n$-valued group $G$ on $X$. In \cite{Yagodovskii}, a sufficient criterion for the integrability of a multivalued dynamics by means of a multivalued group was obtained.

\section
[Algebraic $\texorpdfstring{\boldsymbol{n}}{n}$-Valued Monoids and Groups]
{Algebraic $\boldsymbol{n}$-Valued Monoids and Groups}
\label{algebraic_nval_monoids_and_groups}

To formulate the results of the present paper, we recall and introduce several definitions and constructions.

\begin{defi}\label{def_1}
{An {\it algebraic $n$-valued monoid} is an algebraic variety $X$ with an associative $n$-valued multiplication defined by a rational morphism $X\times X\to \Sym^n(X)$ and with a neutral element (identity) $e\in X$, i.e. $x\ast e$ and $e\ast x$ exist and the equalities $x\ast e = e\ast x = [x, x, ..., x]$ hold for every $x\in X$. An {\it algebraic $n$-valued group} is an algebraic $n$-valued monoid on $X$ with a rational inversion morphism $\inv: X\to X$ such that, whenever $\inv(x)$ exists, the following conditions hold: $x\ast\inv(x)$ and $\inv(x)\ast x$ exist, and $e\in x\ast\inv(x)$, $e\in \inv(x)\ast x$. An algebraic $n$-valued monoid (or group) on $X$ is called {\it regular} if the $n$-valued multiplication $X\times X\to \Sym^n(X)$ is defined on the whole variety $X\times X$ (and the map $\inv$ is defined everywhere on $X$ in the case of a group).} 
\end{defi}  

\begin{example}\label{example_1}
On $\Cc\!P^1 = \Cc\cup\{\infty\}$ there is a structure of a $1$-valued commutative algebraic Hadamard monoid $M_{\mult}(\Cc\!P^1)$ with neutral element $1 = [1:1]$ and multiplication
$$(z_1:z_0)\cdot (w_1:w_0) = (z_1w_1:z_0w_0),$$
defined on $(\Sym^2\Cc\!P^1)\backslash\left\{[0, \infty]\right\}$. The element $\infty = (1: 0)$ is absorbing in this monoid, i.e. $z\ast\infty = \infty$ for every $z\in\Cc\!P^1\backslash\{0\}$. The elements of $\Cc\!P^1\backslash\{0, \infty\}$, and only they, have an inverse $\inv (z_1:z_0) = (z_0:z_1)$.
\end{example}

\begin{example}\label{example_cusp}
On $\Cc\!P^1 = \Cc\cup\{\infty\}$ there is also a structure $M_{\cusp}(\Cc\!P^1)$ of a $1$-valued algebraic commutative monoid with neutral element $0$ and multiplication
$$(z_1:z_0)\cdot(w_1:w_0) = (z_1w_0 + z_0w_1: z_0w_0),$$
defined on $(\Sym^2\Cc\!P^1)\backslash\left\{[\infty, \infty]\right\}$. The elements of $\Cc\!P^1\backslash\{\infty\}$, and only they, have an inverse $\inv (z_1:z_0) = (-z_1:z_0)$. 
\end{example}

\begin{defi}\label{isos_nMon}
Two algebraic (analytic, topological) $n$-valued monoids (groups) $X$ and $Y$ are called {\it isomorphic} if there exists an isomorphism (homeomorphism) $\phi: X\to Y$ inducing a commutative diagram
\begin{equation}\label{iso_square}
\xymatrix{X\times X\ar[rr]\ar[dd]_{\cong} & & \Sym^n(X)\ar[dd]_{\cong}\\ & & \\ Y\times Y\ar[rr] & & \Sym^n(Y)}
\end{equation} 
\end{defi}

\begin{example}
The algebraic $1$-valued monoids $M_{\mult}(\Cc\!P^1)$ and $M_{\cusp}(\Cc\!P^1)$ are not isomorphic, since $M_{\mult}(\Cc\!P^1)$ contains elements of finite order.
\end{example}  

\begin{example}\label{Z_plus_and_G_2}
Consider the discrete coset group from \cite[Section 4]{Buchstaber} on the set $\Z_{+}$ of nonnegative integers, with identity $0$ and multiplication
$$x_1\ast x_2 = \ [x_1 + x_2, |x_1 - x_2|].$$
{Let us compare it with the coset subgroup $\G_2(\mathrm{Sqr}(\Z))\subset \G_2(\Cc)$, where $\mathrm{Sqr}(\Z)$ denotes the set of squares of integers; it has identity $0$ and multiplication
$$y_1\ast y_2 = [(\sqrt{y_1} + \sqrt{y_2})^2, (\sqrt{y_1} - \sqrt{y_2})^2]$$
for all nonnegative integers $y_1$ and $y_2$.} The squaring map $x\mapsto x^2$ is a bijection between the corresponding orbit spaces and makes the square (\ref{iso_square}) commutative. Therefore these two-valued groups are isomorphic.
\end{example}

\begin{defi}\label{sym_n-algebraic_n-valued_definition_first}
A {\it symmetric n-algebraic n-valued monoid {\normalfont(}group{\normalfont)} on \(\Cc\)} is an algebraic $n$-valued monoid {\normalfont(}group{\normalfont)} $\G_{\Cc}(\,f(z; x,y))$ with a partially defined multiplication
$$x\ast y = [z\mid f(z; x, y) = 0],$$
where the polynomial $f(z;(-1)^nx,(-1)^ny)$ is symmetric and has degree at most $n$ in each variable.
\end{defi}

\begin{remark*}
In this definition, neither the map $\inv$ nor the neutral element is specified.
\end{remark*}

\begin{example}\label{universal_2_valued}
Let
\begin{equation}\label{B_a}
B_{\boldsymbol{a}}(z; x, y) = (x + y + z - a_2xyz)^2 - 4(1 + a_3xyz)(xy + yz + xz + a_1xyz)
\end{equation}
be the Buchstaber polynomial. Equivalently, in elementary symmetric functions:
$$B_{\boldsymbol{a}}(z; x, y) = e_1^2 - 4 e_2 - 4 a_1 e_3 - 2 a_2 e_1 e_3 - 
 4 a_3 e_2 e_3 + (a_2^2 - 4 a_1 a_3) e_3^2.$$
The polynomials $B_{\boldsymbol{a}}(z; x, y)$ define on $\Cc$ the structure of the universal algebraic two-valued group $\G_{\Cc}(B_{\boldsymbol{a}})$ for any $\boldsymbol{a}\in\Cc^3$, with multiplication
$$x\ast y = [z\mid B_{\boldsymbol{a}}(z; x, y) = 0],$$
neutral element $0$, and map $\inv(x) = x$ \cite[Theorem 6.3]{Buchstaber_Veselov19}, \cite[Theorem 1]{BK}.
\end{example}

\begin{defi}
Let $M$ be a single-valued monoid and let $H$ be a subgroup of order $n$ of the automorphism group $\Aut(M)$. We shall call the result of applying the construction from \cite[Theorem 1]{Buchstaber} to $M$ and $H$ a {\it coset} $n$-valued monoid, and denote it by $\M(H)$. 
\end{defi}

\begin{prop}
The definition of a coset monoid is correct.
\end{prop}

\begin{proof}
Let $\pi: M\to X = M/H$ be the projection onto the orbit space. Let $\pi(m_1) = x_1$ and $\pi(m_2) = x_2$. Then the multiplication is given as follows:
$$\begin{aligned}
x_1\ast x_2 &= [\pi(\phi(m_1)\cdot \psi(m_2))\mid \phi, \psi\in H]\\
&= [\pi\phi(m_1\cdot \phi^{-1}\psi(m_2))\mid \phi, \psi\in H]\\
&= [\pi(m_1\cdot \zeta(m_2))\mid \zeta\in H].
\end{aligned}$$
For associativity, on the one hand we have
$$\begin{aligned}
(x_1\ast x_2)\ast x_3 &= [\pi(m_1\cdot \phi(m_2))\ast x_3 \mid \phi\in H]\\
&= [\pi(m_1\cdot \phi(m_2)\cdot \psi(m_3))\mid \phi, \psi \in H].
\end{aligned}$$
On the other hand,
$$\begin{aligned}
x_1\ast(x_2\ast x_3) &= [x_1\ast\pi(m_2\cdot \phi(m_3))\mid \phi\in H]\\
&= [\pi(m_1\cdot\psi(m_2\cdot\phi(m_3)))\mid \phi, \psi\in H]\\
&= [\pi(m_1\cdot \psi(m_2)\cdot \psi(\phi(m_3)))\mid \phi, \psi\in H].
\end{aligned}$$
The identity is the class $eH$:
$$x\ast eH = eH\ast x = [m\cdot \phi(e)\mid \phi\in H] = [m, ..., m].$$
\end{proof}

\begin{example}[(The two-valued Chebyshev monoid)]\label{Chebyshev}
Consider on $M_{\mult}(\Cc\!P^1)$ (Example \ref{example_1}) the involution $\tau: z\mapsto 1/z$ for each $z\in \Cc\!P^1$. The points of the orbit space with respect to the involution $\tau$ are represented by the fibers of the branched two-sheeted covering
\begin{equation}\label{cheb_projection}
\begin{aligned}
\pi: \Cc\!P^1&\to \Cc\!P^1\\
z&\mapsto\frac{1}{2}\left(z + \frac{1}{z}\right)
\end{aligned}
\end{equation}
with branch points $\pm 1$. The corresponding coset monoid $\M_{\mult}(\Cc\!P^1) := M_{\mult}(\Cc\!P^1)_{\langle \tau \rangle}$ has identity $1$ and multiplication
\begin{equation}\label{cosine_law}
x\ast y = \left[xy \pm\sqrt{(x^2 - 1)(y^2 - 1)}\right],
\end{equation}
defined on $(\Sym^2\Cc\!P^1)\backslash\{[\infty, \infty]\}$. The described structure of a two-valued algebraic monoid does not extend to a structure of a two-valued algebraic group. The case of the multiplicative torus and the automorphism $z\mapsto 1/z$ was considered in \cite[Section 7, Example 3]{Buchstaber}. The addition law is given by the roots in $z$ of the polynomial
$$P(z; x, y) = z^2 - 2xy z + x^2 + y^2 -1.$$
In homogeneous coordinates, the multiplication $\Cc\!P^1\times \Cc\!P^1\to\Sym^2(\Cc\!P^1)\cong \Cc\!P^2$ is written as follows:
$$(x_1:x_0)\ast(y_1:y_0) = (x_1^2y_0^2+y_1^2x_0^2-x_0^2y_0^2: -2x_1y_1x_0y_0: x_0^2y_0^2).$$
For $x = \cos\alpha$, $y = \cos\beta$, the multiplication (\ref{cosine_law}) takes the form of the addition formulas for the cosine. Consider the two-valued submonoid $\Tt = \Tt(\Cc)$ of the monoid $\M_{\mult}(\Cc\!P^1)$ generated by taking nonnegative integer powers of the element $\cos\alpha$. Let $T_s = T_s(\cos\alpha) = \cos s\alpha$ be the classical Chebyshev polynomials of the first kind ($s\geqslant 0$). Then
$$T_s\ast T_k = [T_{s+k}, T_{|s - k|}].$$
{The algebraic two-valued monoid $\Tt$ is endowed with the structure of a coset algebraic two-valued involutive group, i.e. the inverse map is uniquely defined and has the form $\inv(x) = x$. The group $\Tt$ is isomorphic to $\G_2(\Cc)$ (see Example \ref{Z_plus_and_G_2}). For this reason we shall call $\M_{\mult}(\Cc\!P^1)$ the {\it two-valued Chebyshev monoid}.}    
\end{example}

\begin{example}\label{G_n}
Let
\begin{equation}\label{p_n}
p_n(z; x, y) = \prod\limits_{r, s = 1}^{n}(\sqrt[n]{z} + \varepsilon^r\sqrt[n]{x} + \varepsilon^s\sqrt[n]{y})
\end{equation}
where $\eps$ is some primitive $n$-th root of unity, and $\sqrt[n]{-}$ denotes some complex branch of the root. Then the polynomial
$$p_n(z; (-1)^nx, (-1)^ny)$$
defines the commutative algebraic $n$-valued group $\G_n(\Cc)$ on $\Cc$ from Example \ref{G_n_defi}. The group $\G_n(\Cc)$ is obtained by the coset construction $\G_n(\Cc) = \Cc_{\langle \phi \rangle}$ from the additive group $\Cc$ of complex numbers and its automorphism $\phi: z\mapsto \epsilon z$ of order $n$. 

We present the explicit forms of the polynomials $p_n(z; x, y)$ for $n\leqslant 7$ in the basis of elementary symmetric functions $\sigma_j$: 

\begin{equation}\label{p_1-p_7}\begin{aligned} 
p_1 &=\sigma_1,\\
p_2 &=\sigma_1^2 -2^2\sigma_2,\\
p_3 &=\sigma_1^3 -3^3\sigma_3,\\
p_4 &=\sigma_1^4 - 2^3\sigma_1^2\sigma_2 + 2^4\sigma_2^2 - 2^7\sigma_1\sigma_3,\\
p_5 & = \sigma_1^5 - 5^4\sigma_1^2\sigma_3 + 5^5\sigma_2\sigma_3,\\
p_6 &=\sigma_1^6 -2^2 \cdot 3\sigma_1^4\sigma_2 -2\cdot3^4 \cdot17\sigma_1^3\sigma_3 +2^4\cdot 3\sigma_1^2\sigma_2^2\\ &-2^3 \cdot 3^4 \cdot 19\sigma_1\sigma_2\sigma_3 -2^6\sigma_2^3 +3^3 \cdot19^3 \sigma_3^2,\\
p_7 &= \sigma_1^7 - 7^4\sigma_3(5\sigma_1^4 - 2 \cdot 7^2\sigma_1^2\sigma_2 - 7^4\sigma_1\sigma_3 + 7^3\sigma_2^2).
\end{aligned}\end{equation}
 
\end{example}

We introduce definitions of isomorphisms in the category of coset algebraic (topological) $n$-valued monoids.

\begin{defi}\label{isomorphisms}
Let $M$ and $N$ be $1$-valued algebraic (topological) monoids, and let $A\subset\Aut(M)$ and $B\subset\Aut(N)$ be finite subgroups of order $n$. We say that two coset monoids $M_{A}$ and $N_{B}$ are {\it coset-isomorphic} if there exist isomorphisms $\widetilde{\phi}: M\to N$ and $\phi: M/A\to N/B$ making the following diagram commutative:
\begin{equation}\label{cubic_diagram}
\xymatrix@C=3.5em@R=3.5em{
  & M\times M \ar[rr] \ar[dd] \ar[dl]^{\cong}_{\widetilde{\phi}\times\widetilde{\phi}} && \Sym^n(M) \ar[dd] \ar[dl]^{\cong} \\
  N\times N\ar[rr] \ar[dd] && \Sym^n(N) \ar[dd] \\
  & M/A\times M/A \ar[rr] \ar[dl]^{\cong}_{\phi\times\phi} && \Sym^n(M/A) \ar[dl]^{\cong} \\
  N/B\times N/B \ar[rr] && \Sym^n(N/B)
}
\end{equation}
Moreover, all arrows connecting the front and back faces of the parallelepiped must be isomorphisms. 
\end{defi}

\begin{example}
{As noted in \cite[Proposition 3.2]{BuchRees}, there is an example of two isomorphic two-valued coset groups ${G_{1}}_{H_1}$ and ${G_{2}}_{H_2}$ which are not coset-isomorphic. Namely, consider the infinite dihedral group $G_1 = \langle a,b\mid a^2 = b^2 = e \rangle$ and the group $H_1 = \langle \sigma \rangle\subset\Aut(G_1)$ generated by the involution $\sigma: a\mapsto b$. Let $G_2 = \Z$, and let $H_2 = \langle \iota \rangle\subset \Aut(G_2)$ be the group generated by the involution $\iota: x\mapsto -x$. Then the two-valued groups ${G_1}_{H_1}$ and ${G_2}_{H_2}$ are isomorphic to the two-valued group $\G_2(\mathrm{Sqr}(\Z))$ from Example \ref{Z_plus_and_G_2}, but ${G_1}_{H_1}$ and ${G_2}_{H_2}$ are not coset-isomorphic, since $G_1$ is noncommutative whereas $G_2$ is commutative.}
\end{example}

We also recall the following.
 
\begin{defi}\label{involutive}
An $n$-valued group $G$ is called {\it involutive} if
$$\inv(x) = x$$
for every $x\in G$ {and the map $\inv$ is uniquely defined}.
\end{defi}

\begin{example}
The group $\G_{\Cc}(B_{\boldsymbol{a}})$ from Example \ref{universal_2_valued} is involutive ({the uniqueness of the map $\inv$ follows from the fact that the group $\G_{\Cc}(B_{\boldsymbol{a}})$ is coset; see Section \ref{2_valued_structures_on_CP_1}}).
 \end{example}
 
\begin{defi}\label{iterating_element_definition}
An element $w$ of an $n$-valued monoid $M$ will be called {\it iterating} if, for every $m\in M$, each of the multisets $w\ast m$ and $m\ast w$ contains an element $u$ (depending, in general, on $m$) with multiplicity at least $2$. For $n = 2$, an iterating element $w$ will be called a {\it doubling} element.
\end{defi}

Recall (see \cite[Lemma 1]{Buchstaber}) that the $n$-diagonal construction, or briefly the $n$-diagonal, of a $1$-valued monoid (group) $G$ is the $n$-valued monoid (group) $\diag(G)$ in which $g_1\ast g_2 = [g_1g_2, ..., g_1g_2]$ for any $g_1, g_2 \in G$. Similarly, the $n$-diagonal $\diag_n(G)$ is defined for any $m$-valued monoid (group) $G$; it is an $mn$-valued monoid (group). 

\begin{prop}\label{doubling_monoid}
The set $\W$ of doubling elements of a two-valued monoid $($respectively, group$)$ $\M$ forms a diagonal two-valued submonoid $($respectively, subgroup$)$.
\end{prop}

\begin{proof}
Let $\mathbb{W}$ denote the subset of $\M$ consisting of doubling elements. Let $w_1, w_2\in\mathbb{W}$ and let $w_1\ast w_2 = [w_3, w_3]$ for some $w_3\in\M$. Let $m\in\M$ be an arbitrary element. We have
$$\begin{aligned}
[m\ast w_3, m\ast w_3] &= m\ast (w_1\ast w_2)\\
&= (m\ast w_1)\ast w_2.
\end{aligned}$$
Hence $w_3\in \mathbb{W}$, since the multiset $(m\ast w_1)\ast w_2$ is some quadruple point. The case of $w_3\ast m$ is considered similarly. Consider on the set $\mathbb{W}$ the operation
\begin{equation}
\begin{aligned}\label{operation_doubling_submonoids}
\mathbb{W}\times \mathbb{W}&\to \mathbb{W}\\
w_1\cdot w_2 &= w_3.
\end{aligned}
\end{equation}
It is easy to see that the two-valued submonoid $\mathbb{W}\subset\M$ is the two-diagonal of a single-valued monoid $\W$ with operation (\ref{operation_doubling_submonoids}).

{Now let $\M$ be a two-valued group and let $x\in \W$ be a doubling element. We show that $\inv(x)$ is also doubling. For every $y\in\M$:
\begin{equation}\label{invx_is_iterating_element}
x\ast(\inv(x)\ast y) = (x\ast \inv(x))\ast y = [e,e]\ast y =  [y,y,y,y].
\end{equation}
Let $\inv(x)\ast y = [a, b]$ for some $a, b\in \M$. Then from \eqref{invx_is_iterating_element} it follows that
\begin{equation}\label{ax_equals_xb}
x\ast a = x\ast b = [y,y].
\end{equation}
Multiplying equality \eqref{ax_equals_xb} on the left by $\inv(x)$, we obtain
\begin{equation}\label{invxxa_euqals_invxxb}
(\inv(x)\ast x)\ast a = (\inv(x)\ast x)\ast b.
\end{equation}
Since $\inv(x)\ast x = [e,e]$, it follows from \eqref{invxxa_euqals_invxxb} that $a = b$. Hence $\inv(x)\in\W$.}
      
\end{proof}

\begin{defi}
We shall call the single-valued monoid (group) from the proof of Proposition \ref{doubling_monoid} the monoid (group) of {\it doubling points}.
\end{defi}

It is clear that under isomorphisms of two-valued monoids (groups), the single-valued monoids (groups) of doubling points are preserved.

\begin{example}\label{doubling_elements_Chebyshev}
The coset two-valued algebraic Chebyshev monoid from Example \ref{Chebyshev} has exactly two doubling elements, at the branch points ($1$ and $-1$) of the branched two-sheeted covering \eqref{cheb_projection}. Indeed, let $y$ be an iterating element. {We may assume that $y\neq \infty$, since $\infty\ast\infty$ is not defined in $\M_{\mult}(\CP^1)$ and, consequently, $\infty$ cannot be a doubling element.} Then the polynomial
$$P_{\mult}(z; x, y) = z^2 - 2xy z + x^2 + y^2 -1$$
has a multiple root in $z$ for every $x$, i.e.
\begin{equation}\label{P_mult}
\disc_z(P_{\mult}(z; x, y))  = 4 \left(x^2-1\right) \left(y^2-1\right) = 0.
\end{equation}
Hence $y \in \{\pm 1\}$. The resulting single-valued group of doubling points is isomorphic to $\Z/2$.
\end{example}

\section
[Cubics and $\texorpdfstring{\boldsymbol{n}}{n}$-Valued Coset Addition Laws]
{Cubics and $\boldsymbol{n}$-Valued Coset Addition Laws}
\label{Cubics_and_coset_nval_laws}

In this section we describe the polynomials defining all possible coset addition laws, up to isomorphism, in algebraic $n$-valued monoids and groups on $\Cc\!P^1$ modeled by means of cubics. We introduce and recall several definitions and constructions.

Let
$$\delta_{\boldsymbol{a}} = \disc_t(t^3 + a_1t^2 + a_2t + a_3)$$
be the discriminant with respect to the variable $t$. Then
\begin{equation}\label{Delta_a}
\delta_{\boldsymbol{a}} = -4 a_3 a_1^3+a_2^2 a_1^2+18 a_2 a_3 a_1-4 a_2^3-27 a_3^2.
\end{equation}
Let $\E$ be an irreducible cubic over the field $\Cc$. As is well known (see, for example, \cite[Exercise 5-24]{Fulton}), $\E$ is isomorphic to a cubic in $\Cc\!P^2$ with coordinates $(\zeta:\eta:\xi)$, given in the chart $\{\xi = 1\}$ by the equation
\begin{equation}\label{ellitpic_equation}
\E = \{\eta^2 = \zeta^3 + a_1\zeta^2 + a_2\zeta +a_3\}.
\end{equation} 

An elliptic curve, as an abelian variety over the field $\Cc$, admits automorphisms only of orders $2$, $3$, $4$, and $6$ \cite[Corollary 4.7]{Hartshorne}. We shall successively consider the cases of a nonsingular and a singular irreducible cubic $\E$, and the arising structures of two-, three-, four-, and six-valued groups and monoids.

\subsection{The Case of a Nonsingular Cubic}\label{The_case_of_non-singular_cubic}

Let the point $\boldsymbol{a} = (a_1,a_2,a_3)$ not lie in the singular locus $\{\delta_{\boldsymbol{a}} = 0\}$. Let the complex parameters $\alpha, g_2$, and $g_3$ be such that, in terms of them, the curve $\E$ can be rewritten in the form
$$\eta^2 = (\zeta + \alpha)^3 -\frac{g_2}{4}(\zeta + \alpha) - \frac{g_3}{4},$$
\begin{equation}\label{a1a2a3_system}
\begin{cases}
a_{1}= & 3\alpha,\\
a_{2}= & 3\alpha^{2}-g_{2}/4,\\
a_{3}= & \alpha^{3}-g_{2}\alpha/4-g_{3}/4.
\end{cases}
\end{equation}
Recall that the addition law
$$(\zeta_1, \eta_1)\oplus (\zeta_2, \eta_2) = (\zeta_3, \eta_3)$$
on it is given by the equalities (for distinct points of the curve $\E$):
\begin{equation}\label{sum_formula_x3}
\left\{\begin{aligned} 
\zeta_3 &= -\zeta_1 - \zeta_2 - 3\alpha + \left(\frac{\eta_1 - \eta_2}{\zeta_1 - \zeta_2}\right)^2,\\ 
\eta_3 &= (\zeta_1 - \zeta_3)\cfrac{\eta_1 - \eta_2}{\zeta_1 - \zeta_2} - \eta_1.
\end{aligned}\right.
\end{equation}
When the two summands coincide, equality (\ref{sum_formula_x3}) should be understood in the limiting sense as $\zeta_2\to \zeta_1$.

\subsubsection{Two-Valued Structures on $\Cc\!P^1$}\label{2_valued_structures_on_CP_1} 

There is a branched double covering
\begin{equation}\label{branched_covering}
\begin{aligned}
\pi_2:\E&\to \Cc\!P^1\\
(\zeta,\eta)&\mapsto \zeta,\\
\infty&\mapsto\infty,
\end{aligned}
\end{equation}
{with the following four branch points, each of index two: $\infty$ and $(\zeta_k, 0)$, $k=1,2,3$, where $\zeta_k$ are the roots of the polynomial $\zeta^3 + a_1\zeta^2 + a_2\zeta + a_3$.} The fibers of the map $\pi_2$ are in bijective correspondence with the points of the orbit space $\E/\langle \sigma \rangle$ with respect to the involution
\begin{equation*}
\begin{aligned}\label{involution_sigma}
\sigma:(\zeta,\eta)&\mapsto (\zeta,-\eta),\\
\infty&\mapsto \infty.
\end{aligned}
\end{equation*}
Applying the coset construction \cite[Theorem 1]{Buchstaber} to the involution $\sigma$ on the group of points of the elliptic curve, we obtain a structure $\E_{\langle\sigma \rangle}$ of a coset algebraic two-valued group on $\Cc\!P^1$ with neutral element at the point $\infty$:
\begin{equation}\label{x1_x2_2-valued}
\zeta_1 \ast \zeta_2 = \left [-\zeta_1 - \zeta_2 - 3\alpha + \left(\frac{\eta_1 \pm \eta_2}{\zeta_1 - \zeta_2}\right)^2\right ].
\end{equation}

\begin{prop}\label{lemma_1}
The values of expression {\normalfont{(\ref{x1_x2_2-valued})}} are the roots of the following quadratic polynomial $D(z; \zeta_1, \zeta_2)$ in the variable $z$:
\begin{equation}\label{D_polynomial}
D(z; \zeta_1, \zeta_2) = \Theta_0(\zeta_1, \zeta_2)z^2 + \Theta_1(\zeta_1, \zeta_2)z + \Theta_2(\zeta_1, \zeta_2),
\end{equation}
where
$$\Theta_0=16(\zeta_1-\zeta_2)^2,$$
$$\Theta_1=8\bigl(2g_3+g_2(\zeta_1+\zeta_2+2\alpha)-4\bigl(\zeta_1\zeta_2(\zeta_1+\zeta_2)+6\zeta_1\zeta_2\alpha+3(\zeta_1+\zeta_2)\alpha^2+2\alpha^3\bigr)\bigr),$$
$$
\begin{aligned}
\Theta_2&=(g_2+4\zeta_1\zeta_2)^2+16g_2(\zeta_1+\zeta_2)\alpha
+24\bigl(g_2-4\zeta_1\zeta_2\bigr)\alpha^2\\
&\quad-64(\zeta_1+\zeta_2)\alpha^3-48\alpha^4
+16g_3(\zeta_1+\zeta_2+3\alpha).
\end{aligned}
$$

\end{prop}

\begin{proof}
This is obtained by a direct computation using Vieta's formulas in any computer algebra system, for example, {\tt Wolfram Mathematica}.
\end{proof}

{In \cite{Kontsevich_type_polynomials}, the generalized Kontsevich polynomial
$$D_{\boldsymbol{a}}(z; x, y) = \disc_t\!\bigl(t^3 + a_1 t^2 + a_2 t + a_3 - (t - x)(t - y)(t - z)\bigr),$$
where $\disc_t$ denotes the discriminant with respect to the variable $t$, was introduced, and it was shown that $D_{\boldsymbol{a}}(z; x, y)  = D(-z; -x, -y)$, where $D(z; x, y)$ is the polynomial from Proposition \ref{lemma_1}.}

As we have already noted in the Introduction, part {\bf (i)} of the following theorem was obtained in \cite[Theorem 4.7]{Kontsevich_type_polynomials}. We give a new proof of it using Proposition \ref{lemma_1}, which in turn was obtained by us by methods of computer algebra, without using the theory of elliptic functions. Part {\bf (ii)} is new.

\begin{theorem}\label{2-valued_group}
\leavevmode
\begin{enumerate}[\bf (i)]
\item For $\delta_{\boldsymbol{a}}\neq 0$, the algebraic two-valued group $\G_{\Cc\!P^1}(D_{\boldsymbol{a}})\cong\G_{\Cc\!P^1}(B_{\boldsymbol{a}})$ with identity $\infty$ is an involutive regular coset two-valued group $\E_{\langle\sigma\rangle}$ for the group of points of the elliptic curve
\begin{equation}\label{elliptic_curve_equation_E}
\E = \{\eta^2 = \zeta^3 + a_1\zeta^2 + a_2\zeta + a_3\}
\end{equation}
with respect to the involution $\sigma: (\zeta, \eta)\mapsto (\zeta, -\eta)$. In homogeneous coordinates, the group $\G_{\Cc\!P^1}(B_{\boldsymbol{a}})$ on $\Cc\!P^1$ has identity $(0:1)$ and multiplication
$$\begin{aligned}
\mu_{\boldsymbol{a}}: \Cc\!P^1\times\Cc\!P^1&\to \Cc\!P^2\\
(x_1:x_0)\ast(y_1:y_0) &= (u_2 : u_1 : u_0)
\end{aligned}$$
\begin{align}\label{multiplication_mu}
u_2 &= \left(x_1 y_0-x_0 y_1\right)^2,\nonumber\\
-\frac{1}{2}u_1 &= x_1 x_0 \left(2 a_1 y_1 y_0+a_2 y_1^2+y_0^2\right)+x_1^2 y_1 \left(a_2 y_0+2 a_3 y_1\right)+x_0^2 y_0 y_1,\nonumber\\
u_0 &= x_1^2 y_1 \left(a_2^2 y_1-4 a_3 \left(a_1 y_1+y_0\right)\right)-2 x_0 x_1 y_1 \left(a_2 y_0+2 a_3 y_1\right)+x_0^2 y_0^2.\nonumber
\end{align}

\item {The coset-isomorphism class} of the coset algebraic two-valued group $\G_{\Cc\!P^1}(B_{\boldsymbol{a}})$ is completely characterized by the $j$-invariant of the elliptic curve {\normalfont \eqref{elliptic_curve_equation_E}},
$$j(\boldsymbol{a}) = 6912\,\frac{(3 a_2 - a_1^2)^3}{\,4(3 a_2 - a_1^2)^3 + (27 a_3 - 9 a_1 a_2 + 2 a_1^3)^2\,}.$$.
\end{enumerate}
\end{theorem}

\begin{proof}
\begin{enumerate}[\bf (i)]
\item Express the variables $\alpha$, $g_2$, and $g_3$ from system (\ref{a1a2a3_system}). Substituting them into the formulas from Proposition \ref{lemma_1}, and taking into account the isomorphism
$$\begin{aligned}
\phi: \Sym^2(\Cc\!P^1) &\to \Cc\!P^2\\
[(u_1:u_0), (v_1:v_0)] &\mapsto (u_1v_1:u_1v_0+u_0v_1:u_0v_0),
\end{aligned}$$
we obtain an expression in homogeneous coordinates for the law $\nu_{\boldsymbol{a}}$ defined by the Kontsevich polynomial $D_{\boldsymbol{a}}(-z; -x, -y)$. From the equality
$$B_{\boldsymbol{a}}(z; x, y) = (xyz)^2D_{\boldsymbol{a}}(-1/z; -1/x, -1/y)$$
it follows that, under the Möbius transformation
$$x\mapsto 1/x,\ y\mapsto 1/y,\ z\mapsto 1/z,$$
one obtains the addition formulas $\mu_{\boldsymbol{a}}: \Cc\!P^1\times\Cc\!P^1\to \Cc\!P^2$ in homogeneous coordinates with identity $0$.

Let us show that $\inv(\infty) = \infty$ in $\G_{\Cc\!P^1}(B_{\boldsymbol a})$ for $|a_2|^2 + |a_3|^2\neq 0$. We have in $\G_{\Cc\!P^1}(B_{\boldsymbol a})$:
$$\begin{aligned}
(1:0)\ast(y_1:y_0) = (y_0: -2 (2 a_3 y_1^2+a_2 y_0 y_1):a_2^2 y_1^2-4 a_1 a_3 y_1^2-4 a_3 y_0 y_1).
\end{aligned}$$
For the product $(x_1:x_0)\ast(y_1:y_0) = \phi^{-1}(u_2:u_1:u_0)$ to contain the point $(0:1)$, it is necessary and sufficient that $u_2 = 0$. Hence $y_0 = 0$. Therefore
$$(1:0)\ast(1:0) = (0:-4a_3:a_2^2-4 a_1 a_3).$$
Consequently, $\inv(\infty)$ exists and is equal to $\infty$ if and only if $|a_2|^2 + |a_3|^2\neq 0$.

\item An isomorphism of two-valued groups of the form $\G_{\Cc\!P^1}(D_{\boldsymbol{a}})$ consists of an automorphism of $\Cc\!P^1$ and an isomorphism $\psi$ of abelian varieties $\E_1\to\E_2$ (see Definition \ref{isomorphisms}). As is well known \cite[Lemma 4.9]{Hartshorne}, every morphism of elliptic curves preserving the marked points, namely the neutral elements, is a group homomorphism. The assertion now follows from the fact that the isomorphism class of an elliptic curve is completely characterized by its $j$-invariant \cite[Theorem 4.1]{Hartshorne}.   
\end{enumerate}
\end{proof}

\begin{theorem}\label{doubling_group_for_elliptic_curves}
Let $\E = \{\eta^2 = f(\zeta)\}$ be an elliptic curve, where
$$f(\zeta) = \zeta^3 + a_1\zeta^2 + a_2\zeta + a_3.$$
Then the following assertions hold:
 \begin{enumerate}[\bf (i)]
\item {The doubling elements of the two-valued group $\G_{\Cc\!P^1}(B_{\boldsymbol{a}}(z;$ $ x, y))$ are the point $0$ and the points of the form $1/w$, where $w$ ranges over the set of roots of the polynomial $f(\zeta)$.}   

\item The single-valued group of doubling points of the two-valued group $\G_{\Cc\!P^1}(B_{\boldsymbol{a}})$ is isomorphic to the Klein four-group $\Z/2\times \Z/2$.
 \end{enumerate}
\end{theorem}

\begin{proof}
\begin{enumerate}[\bf (i)]
\item To find all doubling elements for the two-valued group $\G_{\Cc\!P^1}(B_{\boldsymbol{a}}(z;x,y))$, we argue as in Example \ref{Chebyshev} and obtain
\begin{equation}\label{B_discriminant_factorization}
\begin{aligned}
&\disc_z\bigl(B_{\boldsymbol{a}}(z;x,y)\bigr)=\\
&=16xy(a_3x^3+a_2x^2+a_1x+1)(a_3y^3+a_2y^2+a_1y+1)=0
\end{aligned}
\end{equation}
for every \(x\in \Cc\) and fixed \(y\in \Cc\). From \eqref{B_discriminant_factorization} it follows that the point $y=0$, and the points $1/y$ for which $f(1/y)=0$, are doubling. Let us check when the point $\infty$ is doubling. For the discriminant
$\disc_{z_1}\bigl((x_0y_0z_0)^2B_{\boldsymbol{a}}(z_1/z_0; x_1/x_0, y_1/y_0)\bigr)$
at $(y_1:y_0) =\infty =  (1:0)$, we have
\begin{equation}\label{discriminant_at_infty}
16a_3x_1(x_0^3+a_1x_0^2x_1 + a_2x_0 x_1^2 + a_3 x_1^3).
\end{equation}
Expression \eqref{discriminant_at_infty} vanishes for all $x = (x_1:x_0)$ if and only if $a_3 = 0$. Hence the point $\infty$ is iterating $\Longleftrightarrow$ $a_3 = 0$.   

\item It follows from the proof of part {\bf (i)} that the order of the group $\mathbb{W}$ of doubling points is $4$. Since the group \(\G_{\Cc\!P^1}(B_{\boldsymbol{a}})\) is involutive, every nonzero element of $\mathbb{W}$ has order $2$. Hence $\mathbb{W}\cong \Z/2\times \Z/2$.
\end{enumerate}
\end{proof}

From the classification \cite[Theorem 6.3]{Buchstaber_Veselov19} of symmetric two-algebraic two-valued groups on $\Cc$ (see Definition \ref{sym_n-algebraic_n-valued_definition_first} and Example \ref{universal_2_valued}), it follows that every two-algebraic two-valued group on $\Cc$ is defined by the polynomial $B_{\boldsymbol{a}}(z; x, y)$ for some $\boldsymbol{a}$, whose discriminant admits separation of variables:
$$\disc_z\bigl(B_{\boldsymbol{a}}(z;x,y)\bigr)=16\,x^4\,f(1/x)\;\cdot\; y^4\,f(1/y),$$
see formula (\ref{B_discriminant_factorization}). 

Let us note that the first important application of formula (\ref{B_discriminant_factorization}) was obtained by Dragovi\'c in \cite{Dragovic10} (see also \cite{Dragovic14}), on the basis of the discovery of a remarkable connection between the associativity equation of a two-valued group on $\Cc$ and the method of integration of the Kowalevski top.

The separation-of-variables property in the factorization of the discriminant ceases to hold for the $n$-valued laws $p_n(z; x, y)$ (Example \ref{G_n}) already for $n = 3$. For example,
$$\disc_z(p_3(z; x, y)) = -3^9x^2(x-y)^2y^2.$$

\subsubsection{Three-Valued Structures on $\Cc\!P^1$}

There is a unique projective equivalence class of nonsingular cubics whose automorphism group contains an element of order three; in this case the $j$-invariant is equal to $0$. For each complex number $c\neq 0$, the equianharmonic cubic
$$\E = \{\eta^2 = \zeta^3 + c\}$$
belongs to this class \cite[Theorem 3.1.3]{Dolgachev}. Introduce the slope function
$$m = m((\zeta_1, \eta_1), (\zeta_2, \eta_2)) = \frac{\eta_1 - \eta_2}{\zeta_1 - \zeta_2}.$$
Then the addition law for points $(\zeta_1, \eta_1)$ and $(\zeta_2, \eta_2)$ on the elliptic curve $\E$ takes the form
\begin{equation}\label{sum_formula}
\left\{\begin{aligned}
\zeta_3 &= -\zeta_1 - \zeta_2 + m^2,\nonumber\\
\eta_3 &= m(\zeta_1 - \zeta_3) - \eta_1.
\end{aligned}\right.
\end{equation} 

The curve $\E$ admits the automorphism
$$\begin{aligned}
\phi_3: (\zeta, \eta)&\mapsto (\epsilon \zeta, \eta)\\
\infty&\mapsto \infty
\end{aligned}$$
of order $3$ as an abelian variety, where $\epsilon = e^{2\pi i/3}$. Indeed,
$$\begin{aligned}
\phi_3(\zeta_1, \eta_1)\oplus \phi_3(\zeta_2, \eta_2) &= \left(- \epsilon \zeta_1 - \epsilon \zeta_2 + \left(\frac{\eta_1 - \eta_2}{\epsilon \zeta_1 - \epsilon \zeta_2}\right)^2, \frac{\eta_1 - \eta_2}{\epsilon \zeta_1-\epsilon \zeta_2}(\epsilon \zeta_1 - \epsilon \zeta_3) - \eta_1 \right)\\
&= \phi_3(\zeta_3, \eta_3)
\end{aligned}$$   

The orbit $\{(\zeta, \eta),\ (\epsilon \zeta, \eta), (\epsilon^2 \zeta, \eta)\}$ of the automorphism $\phi_3$ bijectively corresponds to the value of $\eta$. We have a branched three-sheeted covering
\[
\begin{aligned}
\pi_3:\E&\to \Cc\!P^1\\
(\zeta,\eta)&\mapsto \eta,\\
\infty&\mapsto\infty.
\end{aligned}
\]
whose base is identified with the orbit space $\E/\langle \phi_3 \rangle$. {The branch points of $\pi_3$ are the three points with ramification index three: $(0, \pm\sqrt{c})$ and $\infty$.} 

We write the three-valued law:
\begin{align}
\eta_1\ast \eta_2&=[\,\pi_3\bigl((\zeta_1,\eta_1)\oplus (\epsilon^k \zeta_2,\eta_2)\bigr)\mid k=0,1,2\,]\nonumber\\
&=\left[\,m_k\!\left(2\sqrt[3]{\eta_1^2-c}+\epsilon^k\sqrt[3]{\eta_2^2-c}-m_k^2\right)-\eta_1\,\right],\label{3_valued_law_multiset}
\end{align}
where for each $k = 0, 1, 2$ we have
$$m_k=\frac{\eta_1-\eta_2}{\sqrt[3]{\,\eta_1^2-c\,}-\epsilon^k\sqrt[3]{\,\eta_2^2-c\,}}.$$

\begin{theorem}\label{3-valued_group}
On $\Cc\!P^1$, for each nonzero $c\in\Cc$, there exists a structure, which we call equianharmonic, of a {commutative regular} algebraic three-valued coset group $\G_{3, c}(\Cc\!P^1)$ with neutral element $0$, inverse map $\inv(x) = -x$, and multiplication defined by the polynomial $p_{3, c}(z; -x, -y)$, where
$$\begin{aligned}
p_{3, c}(z; x, y) &= \sigma_1^3 - 27\sigma_3\\
&+ 18 c \sigma_1^2 \sigma_3 - 54 c \sigma_2 \sigma_3 - 27 c^2 \sigma_2^2 \sigma_3 + 81 c^2 \sigma_1 \sigma_3^2,
\end{aligned}$$
which is a symmetric polynomial, and $\sigma_k$ denotes the $k$-th elementary symmetric function of $x, y, z$. All such three-valued groups are {coset-isomorphic}.
\end{theorem}

\begin{proof}
{A direct computation in a computer algebra system using Vieta's formulas shows that the polynomial $(xyz)^3p_{3, c}(1/z; -1/x,$ $ -1/y)$ has as its roots the elements of the multiset (\ref{3_valued_law_multiset}). The regularity of the three-valued multiplication is established analogously to the proof of Theorem \ref{2-valued_group}.}
\end{proof}

\begin{prop}\label{iterating_3-val}
The set $\{0, \pm1/\sqrt{c}\}$ of iterating elements of the three-valued group $\G_{3, c}(\Cc\!P^1)$ is, as a three-valued subgroup, the diagonal construction of a single-valued group isomorphic to $\Z/3$.
\end{prop}

\begin{proof}
Let $y$ be an iterating element in $\G_{3, c}(\Cc\!P^1)$ and {$y\neq\infty$}. Then for every $x\in\Cc$ we have the vanishing of the discriminant of the polynomial $p_{3, c}(z; x, y)$:
$$-3^9 x^2 y^2 \left(c x^2-1\right)^2 \left(c y^2-1\right)^2 (x-y)^2 \left(9 c^2 x^2 y^2-c x^2+8 c x y-c y^2+1\right)^2 = 0.$$
It follows that the elements $0$, $\pm1/\sqrt{c}$ are iterating. 

{Now let $y=\infty = (1:0)$. For the discriminant of the homogeneous polynomial $p_{3, c}((z_1:z_0); (-x_1:x_0),(1:0))$ we have the expression
$$-19683\,c^{4}x_{0}^{2}x_{1}^{2}\left(-x_{0}^{2}+cx_{1}^{2}\right)^{2}\left(-x_{0}^{2}+9cx_{1}^{2}\right)^{2},$$
which is not identically zero for $c\neq 0$. Hence $\infty$ is not an iterating element.}  

Let $w = 1/\sqrt{c}$. We have the following multiplication table:
$$\begin{aligned}
 w\ast w = [-w, &-w, -w],\quad (-w)\ast(-w) = [w,w,w],\\
 -&w\ast w = w\ast (-w) = [0, 0, 0].
\end{aligned}$$
This shows that the iterating elements are endowed with a group structure isomorphic to the group $\diag(\Z/3)$.
\end{proof}

\subsubsection{Four-Valued Structures on $\Cc\!P^1$}

There is a unique projective equivalence class of nonsingular cubics whose automorphism group is isomorphic to $\Z/4$; in this case the $j$-invariant is equal to $1728$. This class contains the harmonic cubics \cite[Theorem 3.1.3]{Dolgachev}
$$\E = \{\eta^2 = \zeta^3 + b\zeta\}.$$
Consider the map
\[
\begin{aligned}
\phi_4:\ \E&\to \E\\
(\zeta,\eta)&\mapsto (-\zeta,\, i \eta)\\
\infty &\mapsto\infty.
\end{aligned}
\]
It is easy to see that $\phi_4$ is an automorphism of the abelian variety $\E$. The orbit space $\E/\langle \phi_4 \rangle$ is identified with the fibers of the projection
$$\begin{aligned}
\pi_4: \E&\to \Cc\!P^1,\\
(\zeta, \eta)&\mapsto \zeta^2,\\
\infty &\mapsto\infty.
\end{aligned}$$
{The branch points of $\pi_4$ are the two points of order four, $(0,0)$ and $\infty$, and the two points of order two, $(\pm\sqrt{-b}, 0)$.} 

Let $x = \zeta_1^2$, $y = \zeta_2^2$ be points of $\CP^1$ for some $\zeta_1,\zeta_2\in\CP^1$. We write the four-valued addition law:
\begin{equation}
\begin{aligned}
\label{4_valued_law_multiset}
x\ast y
&=\bigl[\,\pi_4\bigl((\zeta_1,\eta_1)\oplus\phi_4^{\,r}(\zeta_2,\eta_2)\bigr)\mid r=0,\dots,3\,\bigr]\\
&=\left[\left(-\sqrt{x}-(-1)^{\ell}\sqrt{y}+m_{\ell,k}^{\,2}\right)^{\!2}\ \middle|\ \ k, \ell = 0,1\right],
\end{aligned}
\end{equation}
where
$$m_{\ell,k}
=\frac{\sqrt{\sqrt{x^{3}}+b\sqrt{x}}-(-1)^k\cdot\sqrt{(-1)^{\, \ell}\bigl(\sqrt{y^{3}}+b\sqrt{y}\bigr)}}{\sqrt{\,x}-(-1)^{\,\ell}\sqrt{\,y}}.$$

\begin{theorem}\label{4-valued_group}
On $\Cc\!P^1$, for each nonzero $b\in\Cc$, there exists a structure, which we call harmonic, of a commutative involutive {regular} algebraic four-valued coset group $\G_{4, b}(\Cc\!P^1)$ with neutral element $0$ and multiplication defined by the symmetric polynomial
\begin{align*}
p_{4, b}(z; x, y) &= \sigma_{1}^{4}
- 8\, \sigma_{1}^{2} \sigma_{2}
+ 16\, \sigma_{2}^{2}
- 128\, \sigma_{1} \sigma_{3}\\
&- 112\, b\, \sigma_{1}^{2} \sigma_{3}
- 4\, b^{2} \sigma_{1}^{3} \sigma_{3}
- 64\, b\, \sigma_{2} \sigma_{3}
- 112\, b^{2} \sigma_{1} \sigma_{2} \sigma_{3}
- 64\, b^{2} \sigma_{3}^{2}\\
&- 288\, b^{3} \sigma_{1} \sigma_{3}^{2}
+ 6\, b^{4} \sigma_{1}^{2} \sigma_{3}^{2}
- 136\, b^{4} \sigma_{2} \sigma_{3}^{2}
- 112\, b^{5} \sigma_{3}^{3}
- 4\, b^{6} \sigma_{1} \sigma_{3}^{3}
+ b^{8} \sigma_{3}^{4},
\end{align*}
where $\sigma_k$ denotes the $k$-th elementary symmetric function of $x, y, z$. All such four-valued groups are {coset-isomorphic.} 
\end{theorem}

\begin{proof}
{A direct computation in a computer algebra system using Vieta's formulas shows that the polynomial $(xyz)^4p_{4, b}(1/z; 1/x, 1/y)$ has as its roots the elements of the multiset (\ref{4_valued_law_multiset}). The regularity of the four-valued multiplication is established analogously to the proof of Theorem \ref{2-valued_group}.}  
\end{proof}

\begin{prop}\label{iterating_4-val}
{The iterating elements of the four-valued group $\G_{4, b}(\Cc\!P^1)$ form the set $\{0, -1/b, \infty\}$, which is a four-valued subgroup with neutral element $0$ and multiplication table
$$\begin{aligned}
\infty\ast\infty &= [0,0,0,0],\quad \infty\ast\left(-\frac{1}{b}\right) = \left [ -\frac{1}{b}, -\frac{1}{b}, -\frac{1}{b}, -\frac{1}{b} \right],\\
&\left(-\frac{1}{b}\right)\ast\left(-\frac{1}{b}\right) = [0,0,\infty,\infty].
\end{aligned}$$
This subgroup is isomorphic to the group $\diag_2((\Z/4)/{\sigma})$, where $\diag_2$ denotes the double diagonal and $(\Z/4)/{\sigma}$ is the two-valued coset group constructed from the additive group $\Z/4$ and the involution $\sigma:a\mapsto -a$ for all $a\in\Z/4$.} 
\end{prop}

\begin{proof}
Let $y$ be an iterating element in $\G_{4, b}(\Cc\!P^1)$ and let $y\neq\infty$. Then for every $x\in\Cc$ we have the vanishing of the discriminant of the polynomial $p_{4,b}(z; x, y)$:
$$\begin{aligned}
&x^3 y^3 (b x+1)^2 (b y+1)^2 (x-y)^2 \left(b^2 x y-1\right)^2 \left(b^2 x^2+4 b^2 x y+2 b x+1\right)^2\\
&\left(b^2 x^2 y+2 b x y+4 x+y\right)^2 \left(4 b^2 x y+b^2 y^2+2 b y+1\right)^2 \left(b^2 x y^2+2 b x y+x+4 y\right)^2 = 0.
\end{aligned}$$
Thus, the elements $0$ and $-1/b$ are iterating elements of the group $\G_{4, b}(\Cc\!P^1)$. 

{Consider the case $y = \infty = (1:0)$. We have the equality in homogeneous coordinates:
$$p_{4,b}((z_1:z_0); (x_1:x_0), (1:0)) = (-x_0z_0 + b^2x_1z_1)^4.$$
It follows that the point $\infty$ is iterating. The explicit multiplication table for the set $\{0, -1/b, \infty\}$ of iterating elements is obtained by direct computations. The isomorphism with the four-valued group $\diag_2((\Z/4)/{\sigma})$ is established by the bijection $0\mapsto 0$, $(-1/b)\mapsto 1$, and $\infty\mapsto 2$.} 
\end{proof}

\begin{remark}
We point out that \cite{K_class} contains a complete classification of $n$-valued monoids and groups of third order. 
\end{remark}

\subsubsection{Six-Valued Structures on $\Cc\!P^1$}
There is a unique projective equivalence class of nonsingular cubics, with $j$-invariant equal to $0$, whose automorphism group is isomorphic to $\Z/6$. For each complex number $c\neq 0$, the equianharmonic cubic
\begin{equation*}
\E = \{\eta^2 = \zeta^3 + c\}
\end{equation*}
belongs to this class \cite[Theorem 3.1.3]{Dolgachev}. Consider the map ($\epsilon = e^{2\pi i/3}$)
$$\begin{aligned}
\phi_6: \E&\to \E\\
(\zeta, \eta)&\mapsto (\epsilon \zeta, -\eta),\\
\infty&\mapsto\infty.
\end{aligned}$$
It is easy to see that $\phi_6$ is an automorphism of the abelian variety $\E$. The points of the orbit space $\E/\langle \phi_6 \rangle$ are identified with the fibers of the projection
$$\begin{aligned}
\pi_6: \E&\to \Cc\!P^1\\
(\zeta, \eta)&\mapsto \eta^2,\\
\infty&\mapsto\infty.
\end{aligned}$$
{The branch points of $\pi_6$ are the two points of index three, $(0,\pm\sqrt{c})$, the three points of index two, $(\epsilon^k\sqrt[3]{-c},0)$, $k=1,2,3$, and the point with ramification index six, $\infty$.}

Let $x = \eta_1^2$, $y = \eta_2^2$ be points of $\CP^1$ for some $\eta_1,\eta_2\in\CP^1$. We write the six-valued addition law:
\begin{equation}
\begin{aligned}\label{6_valued_law_multiset}
&x\ast y = [\pi_6((\zeta_1, \eta_1)\oplus \phi_6^{\, r}(\zeta_2, \eta_2))\mid r = 0, ..., 5] \\
&= \left[\left(m_{\ell, k}\left ( 2\sqrt[3]{\,x - c} + \epsilon^k\sqrt[3]{\,y - c} - m_{\ell, k}^2 \right ) - \sqrt{x}\right)^2\mid \ k = 0, 1, 2; \ell = 0, 1\right],
\end{aligned}
\end{equation}
where
$$m_{\ell, k} = \frac{\sqrt{x} - (-1)^{\ell}\sqrt{y}}{\sqrt[3]{\,x - c} - \epsilon^{\, k}\cdot\sqrt[3]{\,y - c}}.$$

\begin{theorem}\label{6-valued_group}
On $\Cc\!P^1$, for each nonzero $c\in\Cc$, there exists a structure, which we call equianharmonic, of a commutative involutive {regular} six-valued algebraic coset group $\G_{6, c}(\Cc\!P^1)$ with neutral element $0$ and multiplication defined by the symmetric polynomial
\begin{align*}
p_{6, c}(z; x, y) &=  \sigma_{1}^{6}
- 2^{2}\cdot 3\, \sigma_{1}^{4} \sigma_{2}
+ 2^{4}\cdot 3\, \sigma_{1}^{2} \sigma_{2}^{2}
- 2^{6} \sigma_{2}^{3}
- 2\cdot 3^{4}\cdot 17\, \sigma_{1}^{3} \sigma_{3}\\
& - 2^{3}\cdot 3^{4}\cdot 19\, \sigma_{1} \sigma_{2} \sigma_{3}
+ 3^{3}\cdot 19^{3} \sigma_{3}^{2}\\
& - 2^{5}\cdot 3^{2}\cdot 11\, c\, \sigma_{1}^{4} \sigma_{3}
- 2^{2}\cdot 3^{3}\cdot 5\, c^{2} \sigma_{1}^{5} \sigma_{3} 
- 2^{4}\cdot 3^{2}\cdot 211\, c\, \sigma_{1}^{2} \sigma_{2} \sigma_{3}\\
& - 2^{2}\cdot 3^{3}\cdot 197\, c^{2} \sigma_{1}^{3} \sigma_{2} \sigma_{3}
- 2^{3}\cdot 3^{5} c^{3} \sigma_{1}^{4} \sigma_{2} \sigma_{3}
- 2^{6}\cdot 3^{2}\cdot 5^{2} c\, \sigma_{2}^{2} \sigma_{3}\\
& - 2^{4}\cdot 3^{3}\cdot 107\, c^{2} \sigma_{1} \sigma_{2}^{2} \sigma_{3}
- 2^{4}\cdot 3^{7} c^{3} \sigma_{1}^{2} \sigma_{2}^{2} \sigma_{3}
- 2\cdot 3^{6} c^{4} \sigma_{1}^{3} \sigma_{2}^{2} \sigma_{3}\\
& - 2^{6}\cdot 3^{5} c^{3} \sigma_{2}^{3} \sigma_{3}
- 2^{3}\cdot 3^{7} c^{4} \sigma_{1} \sigma_{2}^{3} \sigma_{3}
+ 2^{4}\cdot 3^{3}\cdot 7^{2}\cdot 19\, c\, \sigma_{1} \sigma_{3}^{2}\\
& + 2^{2}\cdot 3^{3}\cdot 47\cdot 53\, c^{2} \sigma_{1}^{2} \sigma_{3}^{2}
+ 2^{3}\cdot 3^{5}\cdot 61\, c^{3} \sigma_{1}^{3} \sigma_{3}^{2}
+ 2\cdot 3^{7}\cdot 17\, c^{4} \sigma_{1}^{4} \sigma_{3}^{2}\\
& + 2^{2}\cdot 3^{4}\cdot 701\, c^{2} \sigma_{2} \sigma_{3}^{2}
+ 2^{3}\cdot 3^{6}\cdot 7\cdot 13\, c^{3} \sigma_{1} \sigma_{2} \sigma_{3}^{2}
+ 2^{10}\cdot 3^{6} c^{4} \sigma_{1}^{2} \sigma_{2} \sigma_{3}^{2}\\
& + 2^{4}\cdot 3^{8}\cdot 5\, c^{5} \sigma_{1}^{3} \sigma_{2} \sigma_{3}^{2}
+ 2\cdot 3^{7}\cdot 5\cdot 17\, c^{4} \sigma_{2}^{2} \sigma_{3}^{2}
+ 2^{5}\cdot 3^{8}\cdot 7\, c^{5} \sigma_{1} \sigma_{2}^{2} \sigma_{3}^{2}\\
& + 2^{2}\cdot 3^{9}\cdot 17\, c^{6} \sigma_{1}^{2} \sigma_{2}^{2} \sigma_{3}^{2} 
+ 2^{2}\cdot 3^{9}\cdot 11\, c^{6} \sigma_{2}^{3} \sigma_{3}^{2}
+ 2^{3}\cdot 3^{11} c^{7} \sigma_{1} \sigma_{2}^{3} \sigma_{3}^{2}\\
& + 3^{12} c^{8} \sigma_{2}^{4} \sigma_{3}^{2}
- 2^{3}\cdot 3^{12} c^{3} \sigma_{3}^{3}
- 2\cdot 3^{9}\cdot 5\cdot 47\, c^{4} \sigma_{1} \sigma_{3}^{3}\\
& - 2^{4}\cdot 3^{9}\cdot 17\, c^{5} \sigma_{1}^{2} \sigma_{3}^{3}
- 2^{2}\cdot 3^{9}\cdot 5\, c^{6} \sigma_{1}^{3} \sigma_{3}^{3}
- 2^{6}\cdot 3^{11} c^{5} \sigma_{2} \sigma_{3}^{3}\\
& - 2^{2}\cdot 3^{10}\cdot 5\cdot 11\, c^{6} \sigma_{1} \sigma_{2} \sigma_{3}^{3}
- 2^{3}\cdot 3^{11} c^{7} \sigma_{1}^{2} \sigma_{2} \sigma_{3}^{3}
- 2^{3}\cdot 3^{11}\cdot 5\, c^{7} \sigma_{2}^{2} \sigma_{3}^{3}\\
& - 2\cdot 3^{12} c^{8} \sigma_{1} \sigma_{2}^{2} \sigma_{3}^{3}
- 2^{3}\cdot 3^{12} c^{6} \sigma_{3}^{4}
- 2^{3}\cdot 3^{12} c^{7} \sigma_{1} \sigma_{3}^{4}\\
& + 3^{12} c^{8} \sigma_{1}^{2} \sigma_{3}^{4}
- 2^{2}\cdot 3^{13} c^{8} \sigma_{2} \sigma_{3}^{4},
\end{align*}
where $\sigma_k$ is the $k$-th elementary symmetric function. All such six-valued groups are {coset-isomorphic.}

\begin{proof}
{A direct computation in a computer algebra system using Vieta's formulas shows that the polynomial $(xyz)^6p_{6, c}(1/z; -1/x,$ $ -1/y)$ has as its roots the elements of the multiset (\ref{6_valued_law_multiset}). The regularity of the six-valued multiplication is established analogously to the proof of Theorem \ref{2-valued_group}.}  
\end{proof}

\begin{prop}\label{iterating_6-val}
{The set $\{0, -1/c, \infty\}$ of iterating elements of the six-valued group $\G_{6, c}(\Cc\!P^1)$ is not a six-valued subgroup of any kind {\normalfont(}nor even a submonoid{\normalfont)}.} 
\end{prop}

\begin{proof}
Let $y$ be an iterating element in $\G_{6, c}(\Cc\!P^1)$, {$y\neq \infty$}. Then for every $x\in\Cc$ we have the vanishing of the discriminant of the polynomial $p_{6, c}(z; x, y)$:
$$2^{12}3^{18}x^5 y^5 (c x+1)^4 (c y+1)^4 (x-y)^4 \left(27 c^2 x^3+81 c^2 x^2 y+54 c x^2+18 c x y+27 x+y\right)^4...$$
It follows that the elements $0$ and $(-1/c)$ are iterating.

{Let $y = \infty = (1:0)$. Then
$\disc_{z_1}((x_0y_0z_0)^6p_{6,c}(z_1/z_0; x_1/x_0, y_1/y_0)) \equiv 0$.
Hence $\infty$ is an iterating element.} 

{The value of the polynomial $(x_0y_0z_0)^6p_{6,c}(z_1/z_0; x_1/x_0, y_1/y_0)$ at $x = (1:0)$ and $y = (-1/c:1)$ is $(z_0 + 9 c z_1)^6$. Therefore, the set of iterating elements is not closed under the six-valued multiplication.}
\end{proof}
\end{theorem}

\subsection{The Nodal Case}

We now turn to the consideration of singular cubics. 

\begin{example}
The change of variables $x,y,z\mapsto x+1, y+1, z+1$ shows that the algebraic two-valued monoid $\M_{\mult}(\Cc\!P^1)$ from Example \ref{Chebyshev} is isomorphic to the monoid $\G_{\Cc\!P^1}(B_{\boldsymbol{a}})$ with $a_1 = 1$, $a_2 = a_3 = 0$. This monoid corresponds to the cubic $\{\eta^2 = \zeta^2(\zeta + 1)\}$.
\end{example}

Let $\boldsymbol{a}$ be such that the polynomial
\begin{equation}
P(\zeta) = \zeta^3 + a_1\zeta^2 + a_2\zeta + a_3
\end{equation}
has a root of multiplicity $2$. In this case, the equation of the curve $\E$ takes the form $(\alpha\neq\beta)$:
\begin{equation}\label{nodal_curve}
\eta^2 = (\zeta - \alpha)^2(\zeta - \beta).
\end{equation}
We parametrize $\E$ by the slope $m$ of a line passing through the node $\Oo = (\alpha, 0)$:
$$\left\{\begin{aligned}
\zeta &= m^2 + \beta,\\
\eta &= m\left(m^2 + \beta - \alpha \right ).
\end{aligned}\right.$$

As before, there is a branched double covering
$$\begin{aligned}
\E &\to \Cc\!P^1\\
(\zeta(m), \eta(m))&\mapsto \zeta(m),\\
\infty&\mapsto\infty,
\end{aligned}$$
with branch points at $\alpha$, $\beta$, and $\infty$. 

\begin{lemma}\label{lemma_2}
Let $m_1\oplus m_2 = -m_3$ for the points $m_1, m_2, m_3\in \Cc\!P^1$ of the curve $\E$ with respect to the parametrization under consideration. Then
\begin{equation}\label{m3_node_case}
m_3 = \frac{m_1m_2 - m_{+}m_{-}}{m_1 + m_2},
\end{equation}
where $m_{\pm} = \pm\sqrt{\alpha - \beta}$ and $[m_1, m_2]\neq [m_{+}, m_{-}]$.
\end{lemma}

\begin{proof}
The points $(\zeta_1, \eta_1)$, $(\zeta_2, \eta_2)$, and $(\zeta_3, -\eta_3)$ of the curve $\E$, with slopes $m_j = \eta_j/(\zeta_j - \alpha)$, where $x_j\neq \alpha$ and $j = 1, 2, 3$, such that $m_1\oplus m_2 = -m_3$, lie on one line. Hence
$$\det\left (\begin{matrix} \zeta_1 & \eta_1 & 1\\ \zeta_2 & \eta_2 & 1\\ \zeta_3 & \eta_3 & 1  \end{matrix} \right ) = 0.$$
It follows that
$$(m_1 - m_2)(m_2 - m_3)(m_1 - m_3)(m_1m_2 - m_2m_3 - m_1m_3 -m_{+}m_{-}) = 0.$$
If $m_1\neq m_2$, the required assertion follows immediately, since in this case the points $m_j$ are pairwise distinct; otherwise $m_1$ or $m_2$ would be singular. For the case $m_1 = m_2$, the value of $m_3$ is given by the same formula (\ref{m3_node_case}) by continuity. 
\end{proof}

\begin{prop}\label{E_cong_G_m_mult_case}
The Möbius transformation
\begin{equation}\label{Mobius_transform}
m\mapsto \frac{m + m_{-}}{m + m_{+}}
\end{equation}
establishes an isomorphism of $1$-valued algebraic monoids $\E\cong M_{\mult}(\Cc\!P^1)$.
\end{prop}

\begin{proof}
Let $m_1\oplus m_2 = -m_3$. It is enough to prove the identity
\begin{equation}\label{identity_for_mj}
\frac{m_1 + m_{-}}{m_1 + m_{+}}\cdot \frac{m_2 + m_{-}}{m_2 + m_{+}}\cdot \frac{-m_3 + m_{-}}{-m_3 + m_{+}} = 1.
\end{equation}
Consider the polynomial
$$Q(\zeta) = (\zeta + m_1)(\zeta + m_2)(\zeta - m_3).$$
Let $a = m_{+} = -m_{-}$. By Vieta's formulas and Lemma \ref{lemma_2}, we have
$$Q(\zeta) = \zeta^3 + (m_1 + m_2 - m_3) \zeta^2 - a^2 \zeta - m_1m_2m_3.$$
It is clear that $Q(a) = Q(-a)$, and hence we obtain the desired identity (\ref{identity_for_mj}).
\end{proof}

\begin{example}\label{M_nodal}
Consider the involution $\iota$ on the monoid $\E(\Cc)$ such that $\iota: m\mapsto -m$ for $m\in\Cc$ and $\iota(\infty) = \infty$. Then the orbit space $\Cc\!P^1/\langle \iota\rangle$ is identified with $\Cc\!P^1$ by means of the map
$$\begin{aligned}
\psi: \Cc\!P^1&\to \Cc\!P^1\\
w&\mapsto w^2.
\end{aligned}$$
The monoid $\Cc\!P^1$ and the involution $\iota$ correspond to the coset two-valued algebraic monoid $\M_{\node}(\Cc\!P^1) = \E_{\langle \iota \rangle}$ with operation
\begin{equation}\label{node_law}
m_1\ast m_2 = \left[ \left(\cfrac{\sqrt{m_1}\sqrt{m_2} \pm a}{\sqrt{m_1} \pm \sqrt{m_2}}\right)^2\right]
\end{equation}
defined on the set $\Sym^2(\Cc\!P^1)\backslash\{[a, a]\}$, where $a = \alpha - \beta$. The values of $m_1\ast m_2$ are the roots of the quadratic trinomial
$$(m_1 - m_2)^2z^2 - 2 \left(a^2 (m_1 + m_2)-4 a m_1 m_2+m_1 m_2 (m_1+m_2)\right)z + (a^2 - m_1 m_2)^2.$$

Writing the addition law (\ref{node_law}) in the original coordinates $(\zeta, \eta)$ of the curve $\E$, we obtain the algebraic law
$$\zeta_1\ast \zeta_2 = \left[ \left( \cfrac{\sqrt{\zeta_1-\beta}\sqrt{\zeta_2-\beta}\pm(\alpha-\beta)}{\sqrt{\zeta_1-\beta}\pm\sqrt{\zeta_2-\beta}} \right)^2+\beta\right].$$
The values of $\zeta_1\ast \zeta_2$ are the roots, in $z$, of the symmetric {Kontsevich polynomial}
$$D_{\alpha, \beta}(-z; -\zeta_1, -\zeta_2) = (\zeta_1\zeta_2z)^2B_{\alpha, \beta}(1/z; 1/{\zeta_1}, 1/{\zeta_2}),$$
where
$$\begin{aligned}
&B_{\alpha, \beta}(z; \zeta_1, \zeta_2) =\left(\left(\alpha ^2 \zeta_1 \zeta_2-1\right)^2-4 \alpha  \beta  \zeta_1 \zeta_2 (\alpha  \zeta_1-1) (\alpha  \zeta_2-1)\right)z^2 \\
&-2\left(-2 \beta  \zeta_1 \zeta_2 (\alpha  \zeta_1-1) (\alpha  \zeta_2-1)+\alpha  \zeta_1 \zeta_2 (\alpha  (\zeta_1+\zeta_2)-4)+\zeta_1+\zeta_2\right)z\\
&+(\zeta_1 - \zeta_2)^2.
\end{aligned}$$
In elementary symmetric functions $\sigma_k$:
$$\begin{aligned}
&B_{\alpha, \beta}(z; \zeta_1, \zeta_2) = \\
&=\sigma_1^2+\sigma_1 \sigma_3 \left(-2 \alpha ^2-4 \alpha  \beta \right)+4 \alpha ^2 \beta  \sigma_2 \sigma_3-4 \sigma_2+\sigma_3^2 \left(\alpha ^4-4 \alpha ^3 \beta \right)+\sigma_3 (8 \alpha +4 \beta ).
\end{aligned}$$
From Vieta's formulas it follows that the polynomial $B_{\alpha, \beta}(z; \zeta_1, \zeta_2)$ coincides with the Buchstaber polynomial $B_{\boldsymbol{a}}(z; \zeta_1, \zeta_2)$ for
$$\left\{\begin{aligned}
a_1 &= -2\alpha - \beta,\\
a_2 &= \alpha^2 + 2\alpha\beta,\\
a_3 &= -\alpha^2\beta
\end{aligned}\right.$$
In other words, the polynomial $D_{\alpha, \beta}(-z; -\zeta_1, -\zeta_2)$ is the Kontsevich polynomial $D_{\boldsymbol{a}}(-z; -\zeta_1, -\zeta_2)$ with parameters lying on the singular divisor $\{\delta_{\boldsymbol{a}} = 0\}$. From the projective classification of singular cubics, it follows that in the nodal case (\ref{nodal_curve}), up to isomorphism, there is only one monoid $\M_{\node}(\Cc\!P^1)$.   
\end{example}

Recall that every irreducible nodal cubic in $\Cc\!P^2$ is projectively equivalent to the cubic $\eta^2\xi = \zeta^2(\zeta+\xi)$ \cite[Exercise 5-24]{Fulton}. We formulate the main result of this section, which follows from the above discussion.

\begin{theorem}\label{nodal_monoids}
Let $\alpha$ and $\beta$ be distinct complex numbers, and let $\E$ be the singular cubic given in the affine chart $\{\xi = 1\}\subset\Cc\!P^2$ by the equation
$$\eta^2 = (\zeta - \alpha)^2(\zeta - \beta).$$
Then the addition of points on $\E$ and the involution
\begin{equation}\label{the_involution_sigma}
\begin{aligned}
\sigma: (\zeta, \eta)&\mapsto (\zeta, -\eta),\\
\infty&\mapsto\infty
\end{aligned}
\end{equation}
define a unique structure, independent of the parameters $\alpha$ and $\beta$ up to isomorphism, of a two-valued coset algebraic monoid $\M_{\node}(\Cc\!P^1)$ on $\Cc\!P^1$, with neutral element $\infty$ and operation
$$\zeta_1\ast \zeta_2 = \left[ \left( \cfrac{\sqrt{\zeta_1-\beta}\sqrt{\zeta_2-\beta}\pm(\alpha-\beta)}{\sqrt{\zeta_1-\beta}\pm\sqrt{\zeta_2-\beta}} \right)^2+\beta\right],$$
defined by the {Kontsevich polynomial} $D_{\alpha, \beta}(-z; -\zeta_1,-\zeta_2)$ from Example {\normalfont\ref{M_nodal}}. The element $\alpha$ is absorbing, i.e. $\zeta\ast\alpha = \alpha\ast \zeta = [\alpha, \alpha]$ for every $\zeta\in\Cc\!P^1\backslash\{\alpha\}$. The product $\alpha\ast\alpha$ is not defined. Moreover, the element $0$ has an inverse, namely $\inv(0) = 0$ and $0\ast 0 = \left[\alpha(1-\alpha/(4\beta)), \infty\right]$, if and only if $\alpha\neq 0$.  
\end{theorem}

\begin{prop}
{The group of doubling points for $\M_{\node}(\Cc\!P^1)$ is isomorphic to $\Z/2$ and consists of two points: $\infty$ and $\beta$.}
\end{prop}

\subsection{The Cuspidal Case}

Finally, consider the case of a triple zero:
$$\eta^2 = (\zeta - \alpha)^3.$$
Introduce a parametrization by the slope of a line $m = \eta/(\zeta-\alpha)$ passing through its cusp. Then Lemma \ref{lemma_2} immediately implies the following assertion. 

\begin{lemma}
The addition law on the cubic $\E = \{\eta^2 = (\zeta - \alpha)^3\}$ has the form
\begin{equation}\label{cusp_law}
m_3 = \frac{m_1m_2}{m_1 + m_2}.
\end{equation}
\end{lemma}

We easily obtain the following.

\begin{prop}\label{E_cong_G_m_cups_case}
The Möbius transformation $m\mapsto 1/m$ establishes an isomorphism between the $1$-valued algebraic monoids $\E(\Cc)\cong M_{\cusp}(\Cc\!P^1)$ $($see Example {\normalfont\ref{example_cusp}}$)$.
\end{prop}

\begin{example}\label{example_3}
The set $\M_{\node}(\Cc\!P^1)$ from Example \ref{M_nodal}, in the limit $\alpha\to\beta$, takes the form of the monoid $\M_{\cusp}(\Cc\!P^1)$ with identity $\infty$ and multiplication
\begin{equation}\label{cusp_law_eq}
\zeta_1\ast \zeta_2 = \left[\cfrac{(\zeta_1-\alpha)(\zeta_2-\alpha)}{(\sqrt{\zeta_1-\alpha} \pm \sqrt{\zeta_2-\alpha})^2} + \alpha\right].
\end{equation}
\end{example}

Thus, the coset monoid $\E_{\langle\sigma \rangle} = \M_{\cusp}(\Cc\!P^1)$ on $\Cc\!P^1$, constructed from the curve $\E$ and the involution $\sigma: (x, y)\mapsto (x, -y)$, in this case has identity $\infty$, absorbing element $0$, and addition law (\ref{cusp_law_eq}). 

\begin{prop}
The monoids $\M_{\cusp}(\Cc\!P^1)$ and $\M_2(\Cc\!P^1)$ are isomorphic.
\end{prop}

\begin{proof}
The map
$$x\mapsto \frac{1}{x - \alpha}$$
defines an isomorphism $\M_{\cusp}(\Cc\!P^1)\to\M_2(\Cc\!P^1)$.  
\end{proof}

Recall that every irreducible cuspidal cubic in $\Cc\!P^2$ is projectively equivalent to the cubic $\{y^2z = x^3\}$ \cite[Exercise 5-24]{Fulton}. 

\begin{theorem}\label{cusp_monoids}
Let $\alpha\in\Cc$ and let $\E$ be the singular cubic given in the affine chart $\{z = 1\}\subset\Cc\!P^2$ by the equation
$$y^2 = (x - \alpha)^3.$$
Then the following assertions hold:  
\begin{enumerate}[\bf (i)]
\item The addition of points on $\E$ and the involution \eqref{the_involution_sigma} define on $\Cc\!P^1$ the unique structure, independent of the parameter $\alpha$ up to {coset} isomorphism, of a two-valued coset algebraic monoid $\M_2(\Cc\!P^1)$. 

\item The group from Theorem {\normalfont\ref{3-valued_group}}, in the limit $c\to 0$, becomes the monoid $\M_3(\Cc\!P^1)$. 

\item The group from Theorem {\normalfont\ref{4-valued_group}}, in the limit $b\to 0$, becomes the monoid $\M_4(\Cc\!P^1)$. 

\item The group from Theorem {\normalfont\ref{6-valued_group}}, in the limit $c\to 0$, becomes the monoid $\M_6(\Cc\!P^1)$. 
\end{enumerate}
\end{theorem}

\begin{prop}
{The group of doubling points for $\M_{\cusp}(\Cc\!P^1)$ is trivial.}
\end{prop}

\section{Projective Duality and the Family of Monoids $\M_n(\Cc\!P^1)$}\label{projective_duality_and_the_family_of_monoids}

Consider the curve
\[
X_n=\{\,p_n(z; x, y)=0\,\}\subset \Cc\!P^2,
\]
where $p_n(z;x,y)$ is the polynomial \eqref{p_n} from Example \ref{G_n}.

\begin{prop}\label{prop_1}
The projectively dual curve to $X_n$ has the form
\[
X_n^\vee=\{\,P_{n-1}(w;u,v)=0\,\}\subset (\Cc\!P^2)^\ast,
\]
where
$$P_{n-1}(w;u,v)=(uvw)^{n-1}\,p_{n-1}(w^{-1};u^{-1},v^{-1}).$$
\end{prop}

\begin{proof}
It is not difficult to check that the curve $X_n\subset \Cc\!P^2$ admits the following rational parametrization:
\begin{equation}\label{X_n}
(x,y,z)=\bigl((-s-t)^n,\ (-t)^n,\ s^n\bigr).
\end{equation}
The chart $\{w=1\}$, dual to the chart $\{z=1\}$ in $(\Cc\!P^2)^\ast$, consists of the lines $\{ux+vy+1=0\}$, each of which is determined by a pair $(u,v)$. By the definition of projective duality (see, for example, \cite[Chapter 1, Section 1, Subsection B]{Gelfand}), for any curve $X=(x(t),y(t))$ in the chart $\{z=1\}$, its caustic $X^\vee$ in the chart $\{w=1\}$ has a parametric representation $(u(t),v(t))$ such that the equation of the tangent line to $X$ at the point $(x(t),y(t))$ has the form
\begin{equation}\label{straight_line}
u(t)x+v(t)y+1=0.
\end{equation}
Therefore,
\begin{align}\label{dual_parametrization}
(u(t),v(t))&=\left(\frac{y'(t)}{R(t)},\ \frac{-x'(t)}{R(t)}\right),\\
R(t)&=x'(t)\,y(t)-x(t)\,y'(t).\nonumber
\end{align}
Substituting \eqref{X_n} into \eqref{dual_parametrization}, we obtain, in the chart $\{w=1\}$, the following parametric equation for $X_n^\vee$:
\begin{equation}\label{dual_parametric_X_n}
(u,v)=\bigl((-1-t)^{1-n},\ t^{1-n}\bigr).
\end{equation}
Since for odd $n$ the sets of values of the multivalued functions $u^{\frac{1}{1-n}}$ and $-u^{\frac{1}{1-n}}$ coincide, whereas for even $n$ we have $(-u)^{\frac{1}{1-n}}=-u^{\frac{1}{1-n}}$ for a suitable branch of the root, eliminating $t$ from \eqref{dual_parametric_X_n} gives
\begin{equation}\label{resulting_parametrization}
u^{\frac{1}{1-n}}+v^{\frac{1}{1-n}}=(-1)^{n-1}.
\end{equation}

Put $m=n-1$. Consider the algebraic element
$$\theta_1=\bigl(u^{-\frac{1}{m}}+v^{-\frac{1}{m}}\bigr)^{-m}$$
in the extension
\[
\Q(u,v)\subset \Q(\sqrt[m]{u},\sqrt[m]{v}).
\]
Then the minimal polynomial of the element $\theta_1(u,v)$ is $(uvw)^m\,p_m(w^{-1};u^{-1},v^{-1})$, since the minimal polynomial of the element
$$\theta_2=\bigl(u^{\frac{1}{m}}+v^{\frac{1}{m}}\bigr)^m$$
is $p_m(w;u,v)$ according to the Lemma \ref{trans_extension} below. Since \eqref{resulting_parametrization} defines the curve $X^\vee$ in the chart $\{w=1\}$, we obtain the desired $X_n$-discriminant
$$\Delta_{X_n}=(uvw)^m\,p_m(w^{-1};u^{-1},v^{-1}).$$
\end{proof}

\begin{lemma}\label{trans_extension}
Let $x_1, ..., x_m$ be algebraically independent variables and let $n_1, ..., n_m$ be positive integers. Then the polynomial
\begin{equation}\label{sum_of_Frobenius_cells}
q(\{n_i\}, \{x_j\}; z) =
\chi\bigl(F(\underset{n_1}{\underbrace{0,...,0, -x_1}})\boxplus...\boxplus F(\underset{n_m}{\underbrace{0,...,0, -x_m}}); z\bigr)
\end{equation}
in the variables $z, x_1, ..., x_m$ is irreducible over the field $\Cc$, where
$F(\underset{n_j}{\underbrace{0,...,0, -x_j}})$ is the Frobenius companion matrix of the polynomial $t^{n_j} -x_j$, $\chi(M;z)$ denotes the characteristic polynomial $($in the variable $z)$ of a matrix $M$, $A\boxplus B := A\otimes I_n + I_m\otimes B\in\Mat_{mn\times mn}$ denotes the Kronecker sum of matrices $A\in \Mat_{m\times m}$, $B\in \Mat_{n\times n}$, and $I_n$ is the identity $n\times n$ matrix. As a corollary, the polynomials $p_n(z; x_1, ..., x_m)$ in the variables $z, x_1, ..., x_m$ are irreducible over $\Cc$ for all positive integers $n$.
\end{lemma}

\begin{proof}
Introduce the following algebraic extension $K/L$ of the field $L=\Cc(x_1,\dots,x_m)$ and the element $\theta$:
$$
K=L(\alpha_1,\dots,\alpha_m),
\qquad
\alpha_j^{\,n_j}=x_j,
\qquad
\theta=\alpha_1+\cdots+\alpha_m.
$$
We first show that $[K:L]=n_1... n_m$. 

Let $K_{j-1}:=L(\alpha_1,\dots,\alpha_{j-1})$ for each $j = 1, ..., m$. Since $x_i=\alpha_i^{n_i}$ for $i<j$, we have
$$K_{j-1}=\Cc(\alpha_1,\dots,\alpha_{j-1},x_j,\dots,x_m).$$
The algebraic independence of the elements $\alpha_1,\dots,\alpha_{j-1},x_j,\dots,x_m$ implies that
$$R_{j-1}:=\Cc[\alpha_1,\dots,\alpha_{j-1},x_j,\dots,x_m]$$
is a polynomial ring. 

We show that for each $j = 1, ..., m$ the polynomial $P_j(t) = t^{n_j}-x_j$ in the ring $\mathrm{Frac}(R_{j-1})[t] = K_{j-1}[t]$ is irreducible by Eisenstein's criterion; see \cite[Chapter V, \S 7]{Lang}. To this end, note that $R_{j-1}$ is a unique factorization domain and that $x_j$ is a prime element in it. The element $x_j$ clearly does not divide the leading coefficient of $P_j(t)$, but it divides all its other coefficients. At the same time, $x_j^2$ does not divide the constant term of $P_j(t)$, since $R_{j-1}$ is an integral domain. Thus, by Eisenstein's criterion, the polynomial $P_j(t)$ is irreducible, and therefore $[K_j:K_{j-1}] = n_j$. 

By the tower theorem for field extensions, it follows that $[K:L] = n_1...n_m$. Hence the monomials
$$\alpha_1^{s_1}\cdots \alpha_m^{s_m},\qquad 0\le s_j<n_j,$$
form a basis of the extension $K$ over the field $L$.    

We now show that the element $\theta$ is primitive. 

For every tuple of roots of unity $(\zeta_1,\dots,\zeta_m)$, $\zeta_j^{n_j}=1$, there exists a conjugation automorphism $\sigma_{\zeta}$ of the field $K$ over the field $L$, given by
$$\sigma_{\zeta}(\theta)=\zeta_1\alpha_1+\cdots+\zeta_m\alpha_m.$$
These conjugation automorphisms $\sigma_\zeta$ are pairwise distinct. Indeed, if $\sigma_{\zeta}(\theta)=\sigma_{\eta}(\theta)$ for some distinct $\zeta$ and $\eta$, then
$$\sum_{j=1}^m (\zeta_j-\eta_j)\alpha_j=0.$$
Then there would exist an index $j_0$ such that $\alpha_{j_0}$ is algebraic over
$$\Cc(x_1,\dots,\widehat{x_{j_0}},\dots,x_m).$$
It would follow that the element $x_{j_0}=\alpha_{j_0}^{n_{j_0}}$ is algebraic over
$$\Cc(x_1,\dots,\widehat{x_{j_0}},\dots,x_m),$$
contradicting the algebraic independence of the elements $x_1, ..., x_m$. Thus, the conjugation automorphisms $\sigma_\zeta$ are pairwise distinct.

It follows from the above that the element $\theta\in K$ has at least $n_1...n_m$ pairwise distinct conjugates $\sigma_\zeta(\theta)$. Since $[K:L] = n_1...n_m$, we have $L(\theta) = K$, i.e. $\theta$ is a primitive element of the extension $K/L$.  

The operator of multiplication by $\theta$ on the $L$-vector space $K$, in the basis
$$\{\,\alpha_1^{s_1}\cdots \alpha_m^{s_m}\mid 0\le s_j<n_j\,\},$$
has matrix
$$F(0,\dots,0,x_1)\,\boxplus\,\cdots\,\boxplus\,F(0,\dots,0,x_m).$$
Consequently, its characteristic polynomial is exactly $q(\{n_i\},$ $\{x_j\};z)$. Since $K=L(\theta)$, the same operator in the basis
$$1,\theta,\dots,\theta^{n_1...n_m-1}$$
is represented by the Frobenius companion matrix of the minimal polynomial of the element $\theta$ over $L$. Hence
$$q(\{n_i\},\{x_j\};z)=m_{\theta,L}(z),$$
and therefore the polynomial $q$ is irreducible over $L$.  

Finally, since the polynomial $q$ has leading coefficient $1$ with respect to $z$, any nontrivial factorization of it in the ring $\Cc[x_1,\dots,x_m,z]$ would remain nontrivial in the ring $L[z]$, which is impossible. Hence $q$ is irreducible in $\Cc[x_1,\dots,x_m,z]$ as well.

We know that the polynomial $p_n(z^n; x_1, ..., x_m)$ in the variable $z$ is irreducible over $\Cc$ for each $n\geq1$. This immediately implies the irreducibility of the polynomial $p_n(z; x_1, ..., x_m)$ over $\Cc$.

\end{proof}

\begin{example}\label{X_2}
For $X_2$ we have
$$(u,v)=\left(-\frac{1}{1+t},\ \frac{1}{t}\right)\quad\text{or}\quad \frac{1}{u}+\frac{1}{v}=-1.$$
Taking the projective closure, that is, homogenizing, we obtain that $X^\vee$ is defined by the equation
$$P_1=uvw\,p_1(w^{-1};u^{-1},v^{-1})=(u+v)w+uv=0.$$
\end{example}

\begin{example}
For $X_3$:
$$(u,v)=\left(\frac{1}{(1+t)^2},\ \frac{1}{t^2}\right)\quad\text{or}\quad \frac{1}{\sqrt{u}}+\frac{1}{\sqrt{v}}=1.$$
The curve $X_3^\vee$ is defined by the equation
$$P_2=(uvw)^2\,p_2(w^{-1};u^{-1},v^{-1})=(uv-w(u+v))^2-4uvw^2=0,$$
that is,
$$P_2=(uv)^2+(vw)^2+(uw)^2-2u^2vw-2uv^2w-2uvw^2.$$
\end{example}

In \cite[Chapter 1, Example 2.3]{Gelfand}, it is shown for the family of Fermat curves, with integers $n\ge 2$,
\begin{equation}\label{Fermat_curve}
F_n=\{\,x^n+y^n=z^n\,\},
\end{equation}
that for a given $n$ the dual curve, in the chart $\{w=1\}$, is defined by the equation
$$F_n^\vee=\left\{\,u^{\frac{n}{n-1}}+v^{\frac{n}{n-1}}=1\,\right\}.$$
For $n=3$, this gives the explicit form
$$u^6+v^6+w^6-2u^3v^3-2u^3w^3-2v^3w^3=0.$$

From the proof of Proposition \ref{prop_1} we obtain:

\begin{theorem}\label{Fermat_curve_theorem}
Let $F_n$ be the Fermat curve \eqref{Fermat_curve}, $n\ge 2$. Then the curve dual to it is defined by the equation $\{F_n^\vee(u,v,w)=0\}$, where
$$F_n^\vee(u,v,w)=p_{n-1}(w^n;u^n,v^n).$$
\end{theorem}

We note that these polynomials can be realized as determinants of generalized Wendt matrices \cite[Theorem 4]{BK}.

We now show that projective duality between $X_n$ and $X_n^\vee$ defines a shift operator in a certain family of $n$-valued algebraic monoids $\M_n(\Cc\!P^1)$.

\begin{theorem}\label{M_n_monoid_structure}
The group structure $\G_n(\Cc)$ extends only to the structure of an algebraic $n$-valued coset monoid $\M_n(\Cc\!P^1)$ on $\Cc\!P^1$. Here the point $\infty$ is absorbing, i.e.
$$\infty\ast x=x\ast \infty=[\infty,\infty,\dots,\infty]$$
for every $x\in \Cc\!P^1\backslash\{\infty\}$, while the value $\infty\ast \infty$ is not defined. In homogeneous coordinates, the multiplication
\[
\ast:\ \Cc\!P^1\times \Cc\!P^1\longrightarrow \Cc\!P^n
\]
is given by the formula
\begin{equation}\label{the_law}
(x_1:x_0)\ast (y_1:y_0)=(b_n:b_1:\dots:b_0),
\end{equation}
where $b_j=b_j(x,y)$ is the coefficient of $z_1^{n-j} z_0^{\,j}$ in the homogeneous polynomial
$$(x_0y_0 z_0)^n\ p_n\!\left(\frac{z_1}{z_0};\ (-1)^n\frac{x_1}{x_0},\ (-1)^n\frac{y_1}{y_0}\right)$$
whenever $(x_1:x_0)$ and $(y_1:y_0)$ are not both equal to $(1:0)$.
\end{theorem}

\begin{proof}
Consider the $n$-valued law $p_n(z;(-1)^n x,(-1)^n y)$ on $\Cc$ as the expression of the desired law on $\Cc\!P^1$ in the chart $z=1$:
\[
\begin{aligned}
\mu:\ \Cc\!P^1\times \Cc\!P^1&\to \Sym^n(\Cc\!P^1)\\
(x,y)=((x_1:x_0),(y_1:y_0))&\mapsto [(w_1:1),\dots,(w_n:1)],
\end{aligned}
\]
where $w_1,\dots,w_n$ are the roots, in the variable $z$, of the polynomial $p(z;x,y)$. We identify $\Sym^n(\Cc\!P^1)$ with $\Cc\!P^n$ by means of the isomorphism
\begin{align}\label{identification}
\phi:\ \Sym^n(\Cc\!P^1)&\longrightarrow \Cc\!P^n\cong G(1,n,2)\\
\nonumber
\bm{u}=[(u_{11}\!:\!u_{10}),\dots,(u_{n1}\!:\!u_{n0})]&\longmapsto
\bigl(z_1u_{10}-z_0u_{11}\bigr)\cdots \bigl(z_1u_{n0}-z_0u_{n1}\bigr)=\phi(\bm{u})(z_1:z_0),
\end{align}
under which a point $\bm{u}$ of the symmetric power is sent to the homogeneous form $\phi(\bm{u})(z_1:z_0)$, the product of $n$ linear forms
$$\ell_j(z_1:z_0)=z_1u_{j1}-z_0u_{j0},$$
that is, to a point of the Chow variety $G(1,n,2)$. Then, by Vieta's formulas, the composition $\phi\circ \mu$ gives the desired law \eqref{the_law}.

It is easy to see that
$$b_n=(x_1y_0+(-1)^{n+1}x_0y_1)^n,$$
each $b_j$ is divisible by $(x_0y_0)^j$ for $j=1,\dots,n$, and $b_0=(x_0y_0)^n$. Hence the multiplication \eqref{the_law} is defined for all pairs $(x,y)\in \Cc\!P^1\times \Cc\!P^1$ except for $(\infty,\infty)=((1:0),(1:0))$. Moreover, the element $\infty$ is not invertible.

The associativity of the resulting operation is obvious.
\end{proof}

\begin{theorem}\label{shift_duality_monoids}
Under projective duality, the curve $X_n$ $(n\ge 2)$ is sent to
\[
X_n^\vee=\{\, (uvw)^{n-1}\, p_{n-1}(1/w;1/u,1/v)=0\,\}\subset (\Cc\!P^2)^{\ast}.
\]
The composition of the duality $X_n\mapsto X_n^\vee$ with the subsequent Möbius transformation $(u,v,w)\mapsto (1/u,1/v,1/w)$ defines the shift operation $\M_n(\Cc\!P^1)\mapsto \M_{n-1}(\Cc\!P^1)$ in the family of algebraic $n$-valued monoids.
\end{theorem}

\begin{proof}
This follows from Proposition \ref{prop_1}.
\end{proof}

The following fact was first obtained in \cite[Theorem 2.3]{Gaiur}. A direct proof was given in the recent paper \cite{BK}. We give yet another proof using the theory of projective duality, which clarifies the nature of this result.

\begin{prop}[\cite{Gaiur}]
The discriminant $\disc_t(P)$ of the polynomial
$$P(t)=(z t^{n-1}+y)(1+t)^{n-1}+(-1)^{n-1}x t^{n-1}$$
with respect to the variable $t$, which is a polynomial of degree $4n-6$, is related to $p_n(z;x,y)$ by
$$(-1)^n (n-1)^{2n-2} (xyz)^{n-2}\, p_n(z;x,y)=\disc_t(P)$$
for every $n\ge 2$.
\end{prop}

\begin{proof}
Consider the curve $X_n^\vee$. We already know that it is parametrized by formula \eqref{dual_parametric_X_n}. Then, by the definition of the $X_n^\vee$-discriminant, the curve $X_n^{\vee\vee}$ is an irreducible component of the discriminant of the polynomial obtained by restricting the line \eqref{straight_line} to $X_n^\vee$. That is, in the chart $\{w=1\}$ the curve $X_n^{\vee\vee}$ is the discriminant, with respect to $t$, of the equation
\begin{equation}\label{polynomial_for_discriminant}
\frac{1}{(-1-t)^{n-1}}\cdot x+\frac{1}{t^{n-1}}\cdot y+1=0.
\end{equation}
Taking the projective closure of the polynomial on the left-hand side of \eqref{polynomial_for_discriminant}, we obtain $p_n(z;x,y)$ up to a constant factor. It is easy to see that if $xyz=0$, then for $n\ge 2$ the polynomial $P(t)$ has a multiple root; hence $\disc_t(P)=0$. This means that $\disc_t(P)$ is divisible by some power of the monomial $xyz$. By \cite[Theorem 2.2]{Gaiur}, $\disc_t(P)$ has no other singular components. The required assertion now follows by comparing degrees.
\end{proof}

In connection with Bessel kernels for solutions of Picard--Fuchs differential equations corresponding to the kernel
$$K_n=\sum_{j,k}\binom{\,j+k\,}{k}\frac{x^{\,j} y^k}{z^{\,j+k}},$$
the paper \cite{Gaiur} considered the iterated analogue of the polynomials $p_n(z;x,y)$:
\begin{equation}\label{p_nm_polynomial}
p_{n,m}(z;\bm{x})=\prod_{k_1,\dots,k_m=1}^{n}\bigl(\sqrt[n]{z}+\varepsilon^{k_1}\sqrt[n]{x_1}+\dots+\varepsilon^{k_m}\sqrt[n]{x_m}\bigr).
\end{equation}

The polynomial $p_{n, m}(z; x)$ defines an $m$-ary $n^{m-1}$-valued algebraic operation
$$\mu(x_1, ..., x_m) = [z\mid p_{n, m}(z; x) = 0].$$
Denote by $\Oo_{n, m}(\Cc\!P^1)$ the variety $\Cc\!P^1$ equipped with the operation $\mu$.

Let
$$X_{n,m}=\{\,p_{n,m}=0\,\}$$
be a hypersurface in $\Cc\!P^m$. For integers $n\ge 2$ and $m\ge 2$, put
$$P_{n,m}=(u_1\cdots u_m w)^{n-1}\, p_{n-1, m}(w^{-1};u_1^{-1},\dots,u_m^{-1}).$$
By the same method as in Theorem \ref{shift_duality_monoids}, we obtain the following.

\begin{theorem}\label{theorem_3}
The composition of duality, for $m\ge 2$ and $n\ge 2$,
$$X_{n,m}\mapsto X_{n,m}^\vee=\{\,P_{n,m}=0\,\}\subset (\Cc\!P^m)^{\ast},$$
with the subsequent Möbius transformation
$$(u_1,\dots,u_m,w)\mapsto (1/u_1,\dots,1/u_m,1/w)$$
defines the shift operation
$$\Oo_{n,m}(\Cc\!P^1)\mapsto \Oo_{n-1,m}(\Cc\!P^1)$$
in the family of $m$-ary $n^{m-1}$-valued algebraic structures $\Oo_{n,m}(\Cc\!P^1)$.
\end{theorem}

This result clarifies the assertion of \cite[Theorem 3.2]{Gaiur} concerning the relation between the polynomial $p_{n,m}(z;\bm{x})$ and the discriminant of the homogeneous polynomial
$$P(\bm{u})=(u_1\cdots u_m)^{n-1}\left(z+(-1)^n\left(\sum_{j=1}^m u_j\right)^{n-1}\cdot \sum_{j=1}^m \frac{x_j}{u_j^{\,n-1}}\right),$$
taken with respect to the variables $u_1,\dots,u_m$ in the sense of \cite[Chapter 13]{Gelfand}.

The observation from Theorem \ref{Fermat_curve_theorem} has an iterated analogue. 

\begin{theorem}\label{F_nm_hypersurfaces}
Let $F_{n,m}$ be the Fermat hypersurface
$$F_{n,m} = \{\,x_1^n + x_2^n + ... + x_m^n = z^n\,\}$$
in $\Cc\!P^m$ with coordinates $x_1, ..., x_m, z$. The dual hypersurface is defined by the equation
\begin{equation}\label{F_n_vee}
F_{n, m}^{\vee} = \{\,p_{n-1, m}(w^n; u_1^n, ..., u_m^n) = 0\,\}
\end{equation}
in $(\Cc\!P^m)^{\ast}$ with dual coordinates $u_1, ..., u_m, w$, where $p_{n,m}$ denotes the polynomial \eqref{p_nm_polynomial}.   
\end{theorem}

In \cite[Example 4.16]{Gelfand}, it was observed that the dual hypersurface can be defined by the equation
$$u_1^{\frac{n}{n-1}} + ... + u_m^{\frac{n}{n-1}} = z^{\frac{n}{n-1}}$$
and that this irrational equation can be replaced by a polynomial equation of degree $n(n-1)^{m-1}$. Theorem \ref{F_nm_hypersurfaces} clarifies this observation by giving a concrete realization of \eqref{F_n_vee} as a polynomial equation. A determinantal expression for \eqref{F_n_vee} in the case $n = 2$ and $m = 3$ can be found in \cite[Example 9]{BK}.  

Fermat hypersurfaces play an important role in various problems of algebraic topology and algebraic geometry. Their topology has been studied in various works. For example, each Fermat hypersurface $F_{2, m}$ is diffeomorphic to the homogeneous space $SO(m+1)/(SO(2)\times SO(m-1))$ of oriented planes in $\R^{m + 1}$ \cite[Chapter XI, Example 10.6]{Kobayashi_2}.

\section
[The Laws $\texorpdfstring{\bm{p}_n}{p_n}$ and Discriminants of Field Extensions]
{The Laws $\bm{p}_n$ and Discriminants of Field Extensions}
\label{p_n_laws_and_field_extension_discriminants}

To formulate the following proposition, we need a definition first introduced by Dedekind for algebraic number fields \cite[p.~429]{Dedekind}. We give the general version, following \cite[Lecture 12, Definition 12.5]{number_theory_lectures}:

\begin{defi}\label{discriminant_for_field_extensions}
Let $R$ be a commutative ring with identity, and let $R\subset S$ be a finite extension such that $S$ is a free $R$-module. For any elements $e_1,\dots,e_n\in S$, their {\it discriminant} is defined by the formula
\[
\Delta(e_1,\dots,e_n)=\det\bigl(\Tr_{S/R}(e_ie_j)\bigr)_{ij}\,,
\]
where $\Tr(-)$ denotes the trace of the $R$-linear map $S\to S$ given by multiplication by $e_ie_j$.
\end{defi}

In the case of interest to us, Definition \ref{discriminant_for_field_extensions} reduces to the classical definition of the discriminant of a polynomial.

\begin{lemma}[(Lecture 12, Proposition 12.6 \cite{number_theory_lectures})]
Let $K\subset L$ be a finite separable extension of degree $n$, let $\Omega$ be the normal closure of $L$ over $K$, and let $\sigma_1,\dots,\sigma_n$ be the distinct embeddings $L\hookrightarrow \Omega$ over $K$. Then:
\begin{enumerate}[\bf (i)]
\item For any elements $e_1,\dots,e_n\in L$ we have
$$\Delta(e_1,\dots,e_n)=\det(\sigma_i(e_j))_{ij}^{2}.$$
\item For any $x\in L$ we have
$$\Delta(1,x,\dots,x^{n-1})=\prod_{i<j}\bigl(\sigma_i(x)-\sigma_j(x)\bigr)^{2}.$$
\end{enumerate}
\end{lemma}

Under a change of basis $\bm{e}' = \bm{e}C$, $C\in \Mat_K(n)$, the discriminant changes according to the formula
$$\Delta_{L/K}(\bm e')=\det(C)^{2}\,\Delta_{L/K}(\bm e).$$
In the case where $K$ is the field of fractions of a Dedekind domain $A$, $L/K$ is a finite separable extension, and $B$ is the integral closure of $A$ in $L$, this allows one to define the discriminant $\Delta_{L/K}$ of the extension $L/K$ as the fractional ideal generated by the set
$$\{\,\Delta(\bm e)\mid \text{$\bm e$ is an $A$-basis of the $A$-module $B$}\,\}.$$
In our case, the ring $\Q[x,y]$ is not a Dedekind domain.

\begin{prop}\label{p_n_as_field_extension_discriminant}
For every integer $n\ge 2$, the discriminant of the polynomial $p_n(z;$ $x, y)$ with respect to the variable $z$ coincides with the discriminant $\Delta(1,\theta,\dots,\theta^{n-1})$ for the extension $\Q(x,y)\subset \Q(\theta)$, where $\theta=(\sqrt[n]{x}+\sqrt[n]{y})^{n}$.
\end{prop}

\begin{proof}
Indeed, as already noted, $p_n(z; x, y)$ is the minimal polynomial of the element $\theta=(\sqrt[n]{x}+\sqrt[n]{y})^{n}$.
\end{proof}

\newpage

\cleardoublepage

\phantomsection
\addcontentsline{toc}{section}{References}
\bibliographystyle{mystyle}
\bibliography{data}

\newpage

\vspace{1em}
\noindent
\textit{Victor Buchstaber}\\
Steklov Mathematical Institute of Russian Academy of Sciences\\
Email: \texttt{buchstab@mi-ras.ru}

\vspace{1em}
\noindent
\textit{Mikhail Kornev}\\
Steklov Mathematical Institute of Russian Academy of Sciences\\
Email: \texttt{mkorneff@mi-ras.ru}

\end{document}